\documentclass[12pt,a4paper]{amsart}
\usepackage{amssymb}
\usepackage{hyperref}

\calclayout

\newcommand{\+}{\nobreakdash-}
\renewcommand{\:}{\colon}

\newcommand{\rarrow}{\longrightarrow}
\newcommand{\larrow}{\longleftarrow}
\newcommand{\ot}{\otimes}

\DeclareMathOperator{\Hom}{Hom}
\DeclareMathOperator{\Ext}{Ext}
\DeclareMathOperator{\Spec}{Spec}

\newcommand{\Modl}{{\operatorname{\mathsf{--Mod}}}}

\newcommand{\Qcoh}{{\operatorname{\mathsf{--Qcoh}}}}
\newcommand{\Ctrh}{{\operatorname{\mathsf{--Ctrh}}}}
\newcommand{\Comodl}{{\operatorname{\mathsf{--Comod}}}}

\newcommand{\bb}{{\mathsf{b}}}

\newcommand{\sAb}{\mathsf{Ab}}

\newcommand{\lrarrow}{\mskip.5\thinmuskip\relbar\joinrel\relbar
   \joinrel\rightarrow\mskip.5\thinmuskip\relax}

\newcommand{\bu}{{\text{\smaller\smaller$\scriptstyle\bullet$}}}

\newcommand{\cO}{\mathcal O}

\newcommand{\AB}{\mathrm{AB}}

\newcommand{\sA}{\mathsf A}
\newcommand{\sB}{\mathsf B}
\newcommand{\sC}{\mathsf C}
\newcommand{\sD}{\mathsf D}
\newcommand{\sE}{\mathsf E}
\newcommand{\sF}{\mathsf F}
\newcommand{\sK}{\mathsf K}

\newcommand{\I}{\mathcal I}
\newcommand{\J}{\mathcal J}
\newcommand{\K}{\mathcal K}
\newcommand{\cL}{\mathcal L}
\newcommand{\M}{\mathcal M}
\newcommand{\N}{\mathcal N}
\newcommand{\F}{\mathcal F}
\newcommand{\G}{\mathcal G}
\newcommand{\cH}{\mathcal H}
\newcommand{\A}{\mathcal A}
\newcommand{\B}{\mathcal B}
\newcommand{\C}{\mathcal C}
\newcommand{\Q}{\mathcal Q}

\newcommand{\boZ}{\mathbb Z}

\newcommand{\boR}{\mathbb R}

\newcommand{\Section}[1]{\bigskip\section{#1}\medskip}
\theoremstyle{plain}
\newtheorem{thm}{Theorem}[section]
\newtheorem{prop}[thm]{Proposition}
\newtheorem{lem}[thm]{Lemma}
\newtheorem{cor}[thm]{Corollary}
\theoremstyle{definition}
\newtheorem{rem}[thm]{Remark}
\newtheorem{quest}[thm]{Question}
\newtheorem{ex}[thm]{Example}

\begin{document}

\title{Roos axiom holds for quasi-coherent sheaves}

\author{Leonid Positselski}

\address{Institute of Mathematics, Czech Academy of Sciences \\
\v Zitn\'a~25, 115~67 Prague~1 \\ Czech Republic} 

\email{positselski@math.cas.cz}

\begin{abstract}
 Let $X$ be either a quasi-compact semi-separated scheme, or
a Noetherian scheme of finite Krull dimension.
 We show that the Grothendieck abelian category $X\Qcoh$ of
quasi-coherent sheaves on $X$ satisfies the Roos axiom $\AB4^*$\+$n$:
the derived functors of infinite direct product have finite homological
dimension in $X\Qcoh$.
 In each of the two settings, two proofs of the main result are given:
a more elementary one, based on the \v Cech coresolution, and a more
conceptual one, demonstrating existence of a generator of
finite projective dimension in $X\Qcoh$ in the semi-separated case
and using the co-contra correspondence (with contraherent cosheaves)
in the Noetherian case.
 The hereditary complete cotorsion pair (very flat quasi-coherent 
sheaves, contraadjusted quasi-coherent sheaves) in the abelian category
$X\Qcoh$ for a quasi-compact semi-separated scheme $X$ is discussed.
\end{abstract}

\maketitle

\tableofcontents

\section*{Introduction}
\medskip

 The condition of (existence and) exactness of infinite direct
products in an abelian category $\sA$ is known as ``Grothendieck's
axiom $\AB4^*$\,''~\cite[Section~1.5]{GrToh}.
 It is well known that the categories of quasi-coherent sheaves
$X\Qcoh$ over schemes $X$ do \emph{not} satisfy $\AB4^*$ in general.
 A counterexample of Keller, reproduced by Krause in
the paper~\cite[Example~4.9]{Kra}, shows that the infinite products are
not exact in the category of quasi-coherent sheaves over the projective
line $X=\mathbb P^1_k$ over a field~$k$.
 More generally, if $X$ is a Noetherian scheme with an ample family of
line bundles, then the infinite products are exact in $X\Qcoh$ if and
only if $X$ is affine~\cite[Theorem~1.1]{Kan}.

 The axiom $\AB4^*$\+$n$ ($n\ge0$) was proposed by Roos in his 2006
paper~\cite[Definition~1.1]{Roos} as a natural weakening of~$\AB4^*$.
 An abelian category with infinite products and enough injective
objects is said to satisfy $\AB4^*$\+$n$ if the derived functor of
infinite product $\prod_{\lambda\in\Lambda}^{(i)} A_\lambda$ vanishes
on all families of objects $(A_\lambda\in\sA)_{\lambda\in\Lambda}$
for $i>n$.
 A characterization of Grothendieck abelian categories satisfying 
$\AB4^*$\+$n$ was obtained in the same paper~\cite[Theorem~1.3]{Roos}.
 It is clear from that characterization that existence of a generator
of projective dimension~$n$ in a Grothendieck category~$\sA$ implies
$\AB4^*$\+$n$.

 Some examples of Grothendieck categories satisfying $\AB4^*$\+$n$ were
discussed in~\cite[Section~1]{Roos}, and the minimal possible value
of~$n$ for a given category was computed in some cases.
 In particular, \cite[Section~1.3]{Roos} was dedicated to the categories
of quasi-coherent sheaves on Noetherian schemes $X$, and contained
a full treatment of the case when $X$ is the complement to the closed
point in the spectrum of a Noetherian local ring.
 As an application demonstrating the utility of the $\AB4^*$\+$n$
condition, it was shown in~\cite[Theorem~2.1]{Roos} that the derived
functor $\varprojlim^*$ of inverse limit of sequences of objects
in $\sA$ (indexed by the natural numbers) has homological dimension
at most $n+1$ if $\sA$ is a Grothendieck category satisfying
$\AB4^*$\+$n$.

 Further applications of the $\AB4^*$\+$n$ condition can be found in
the recent preprint of Herbera, Pitsch, Saor\'\i n, and
Virili~\cite{HPSV}.
 According to~\cite[Theorems~2.9 and~4.3]{HPSV}, certain classical
constructions going back to Cartan--Eilenberg and Spaltenstein
allow to explicitly produce homotopy injective (DG\+injective)
resolutions of complexes in an abelian category $\sA$ provided that
$\sA$ has infinite products, enough injective objects, and satisfies
$\AB4^*$\+$n$ (see also~\cite[Theorem~1.3]{HX}).
 This allows one to avoid using the more general but less explicit
small object argument.

 The assertion that the Roos axiom $\AB4^*$\+$n$ holds for the category
of quasi-coherent sheaves $X\Qcoh$ on any quasi-compact semi-separated
scheme $X$ covered by $n+1$ affine open subschemes was first proved in
the (still unpublished) 2009~preprint of Hogadi and
Xu~\cite[Remark~3.3]{HX}.
 A more detailed exposition can be found in~\cite[Theorem~8.27]{HPSV}.
 Other results of~\cite[Theorem~1.1 and Section~3]{HX} established
the Roos axiom for certain classes of stacks.
 However, the results of~\cite{HX} and~\cite{HPSV} are not applicable
to non-semi-separated schemes.

 Recall that, generally speaking, one is not supposed to consider
the derived category of quasi-coherent sheaves on a quasi-compact
quasi-separated scheme~$X$; instead, one should work with the derived
category of sheaves of $\cO_X$\+modules with quasi-coherent cohomology
sheaves.
 The reason is that the derived functors of direct image acting
between the derived categories of quasi-coherent sheaves are not
well-behaved (do not commute with restrictions to open subschemes,
are not compatible with the compositions of morphisms of schemes, etc.)
 There are two cases when the derived category of quasi-coherent
sheaves $\sD(X\Qcoh)$ behaves well: either the scheme $X$ should be
semi-separated, or it should be Noetherian~\cite[Appendix~B]{TT},
\cite[proof of Theorem~7.1.3]{Pcosh}.
 These are the two setting considered in the present paper.
 We assume finiteness of the Krull dimension in the Noetherian case.

 In the first half of this paper
(Sections~\ref{cech-semi-separated-secn}
and~\ref{cech-noetherian-secn}), we work out an elementary approach
based on the \v Cech coresolution.
 We start with discussing the argument of Hogadi--Xu and
Herbera--Pitsch--Saor\'\i n--Virili in the semi-separated case,
and then proceed to prove the Roos axiom for the category of
quasi-coherent sheaves on a Noetherian scheme $X$ of finite Krull
dimension using a more sophisticated version of the same approach.

 In the second half of the paper (Sections~\ref{generator-secn}
and~\ref{co-contra-secn}), we discuss very flat and contraadjusted
quasi-coherent sheaves in the semi-separated case, and implications
of the ``na\"\i ve'' co-contra correspondence theorem in
the Noetherian case.
 The approach based on the co-contra correspondence is applicable
to the semi-separated case as well, but the argument with very flat
quasi-coherent sheaves is simpler.
 Notice that the \v Cech coresolution argument gives a very good
numerical bound in the semi-separated case: the Roos axiom
$\AB4^*$\+$n$ holds for $X\Qcoh$ provided that $X$ can be
covered by $N=n+1$ open affines.
 The very flat argument only proves the condition $\AB4^*$\+$N$ in this
case, but it gives a conceptually stronger result: the Grothendieck
category $X\Qcoh$ admits a generator of projective dimension~$N$.

 The notion of a \emph{very flat} quasi-coherent sheaf on a scheme $X$
was introduced in an early, September~2012, version of the present
author's preprint on contraherent cosheaves~\cite{Pcosh}.
 It was immediately clear from~\cite[Lemma~4.1.1]{Pcosh} that, for
any quasi-compact semi-separated scheme $X$, the category of
quasi-coherent sheaves $X\Qcoh$ has a very flat generator.
 It was also easy to see that the projective dimension of any very
flat quasi-coherent sheaf on $X$ does not exceed the minimal number~$N$
of open subschemes in an affine open covering of~$X$.
 But I~did not realize the importance of the property of a Grothendieck
category to have a generator of finite projective dimension back in
my Moscow years.

 That I~learned from Jan St\!'ov\'\i\v cek sometime around~2016.
 In our joint paper~\cite{PS2}, the reference to very flat
quasi-coherent sheaves and to~\cite[Lemma~4.1.1]{Pcosh} appeared in
this connection in~\cite[last paragraph of Example~6.4]{PS2}.
 A conversation with Manuel Saor\'\i n, who visited us in Prague in
July~2024, convinced me that a paragraph in an example in~\cite{PS2}
was not enough to advertise the fact, which seemed to remain unknown
to people working in the area.
 This motivated me to write the first version of the present paper up.
 Then Simone Virili told me about the \v Cech coresolution argument
and the preprint~\cite{HX}, and I~realized that such an argument can be
made to work for Noetherian schemes of finite Krull dimension as well.
 A co-contra correspondence argument for Noetherian schemes of finite
Krull dimension was subsequently worked out
in~\cite[Section~6.10]{Pcosh}.
 We discuss all these approaches to proving the Roos axiom for
quasi-coherent sheaf categories in this paper.

 In particular, a very abstract and general category-theoretic version
of the co-contra correspondence argument is presented in this paper.
 Let us start with explaining the terminology.
 Abstractly speaking, a ``na\"\i ve co-contra correspondence'' is
a derived equivalence $\sD^\star(\sA)\simeq\sD^\star(\sB)$,
where $\star$ ranges over the derived category symbols
$\bb$, $+$, $-$, and~$\varnothing$, while $\sA$ is an exact
category with exact coproduct functors and $\sB$ is an exact
category with exact product functors.
 See~\cite[Theorems~4.8.1 and~6.8.1]{Pcosh} for examples relevant in
our context.
 For comparison, a non-na\"\i ve co-contra correspondence is
a triangulated equivalence between the \emph{coderived category}
of $\sA$ and the \emph{contraderived category} of~$\sB$;
see~\cite[Theorem~9.5]{Pksurv}
and~\cite[Theorems~6.7.1 and~6.8.2]{Pcosh} for examples.
 We refer to the introduction to the paper~\cite{Pmgm} for
a general discussion of the philosophy of co-contra correspondence.

 In the context of the present paper, it is the ``na\"\i ve'' version
of co-contra correspondence that is relevant.
 Our result in this direction is stated as follows.
 Let $\sA$ and $\sB$ be idempotent-complete exact categories (in
the sense of Quillen) such that $\sA$ has infinite direct product
functors and enough injective objects, while $\sB$ has \emph{exact}
infinite product functors.
 Suppose given a triangulated equivalence between the unbounded
derived categories $\sD(\sA)\simeq\sD(\sB)$ that restricts to
an equivalence of the bounded derived categories $\sD^\bb(\sA)\simeq
\sD^\bb(\sB)$.
 Then there \emph{exists} a finite integer $n\ge0$ such that $\sA$
satisfies $\AB4^*$\+$n$ (in a suitable sense).
 For a more concrete version of this argument leading to a concrete
numerical bound $n=2D+1$ in the particular case of the abelian category
$\sA=X\Qcoh$ of quasi-coherent sheaves on a Noetherian scheme $X$ of
Krull dimension~$D$, we refer the reader
to~\cite[Theorem~6.10.1]{Pcosh}.

\subsection*{Acknowledgement}
 I~am grateful to Jan St\!'ov\'\i\v cek, Manuel Saor\'\i n,
Simone Virili, and Michal Hrbek for illuminating conversations
and correspondence.
 I~also wish to thank an anonymous referee for reading the manuscript
carefully and spotting several unpleasant misprints.
 The author is supported by the GA\v CR project 23-05148S and
the Czech Academy of Sciences (RVO~67985840).

\Section{\v Cech Coresolution Argument for Semi-Separated Schemes}
\label{cech-semi-separated-secn}

 This section contains no new results.
 It is included for the sake of completeness of the exposition and
in order to prepare ground for a more complicated version of the same
argument spelled out in Section~\ref{cech-noetherian-secn}.

\subsection{Adjusted objects} \label{adjusted-objects-subsecn}
 Let $\sE$ be an abelian category with enough injective objects.
 Let $\sA$ be another abelian category and $G\:\sE\rarrow\sA$ be
a left exact functor.
 We will denote by $\boR^*G\:\sE\rarrow\sA$ the sequence of right
derived functors of the functor~$G$ (computed, as usual, in terms
of injective coresolutions of objects of~$\sE$).

 The following lemma is standard.

\begin{lem} \label{acyclic-models}
 Let $E\in\sE$ be an object, and let\/ $0\rarrow E\rarrow K^0\rarrow
K^1\rarrow K^2\rarrow\dotsb$ be a coresolution of $E$ by objects
$K^i\in\sE$ such that\/ $\boR^mG(K^i)=0$ for all $i\ge0$ and $m\ge1$.
 Then there are natural isomorphisms
$$
 \boR^iG(E)\simeq H^iG(K^\bu), \qquad i\ge0
$$
of objects in\/~$\sA$.
\end{lem}

\begin{proof}
 This is provable by induction using the long exact sequences of
the derived functors $\boR^*G$ associated with short exact sequences
of objects in~$\sE$.
 See, e.~g., \cite[Exercise~2.4.3]{Wei}.
\end{proof}

\subsection{Direct images under flat affine morphisms}
 Let\/ $\sA$ be an abelian category with infinite products and enough
injective objects.
 Let $\Lambda$ be a set.
 We are interested in the case when $\sE=\sA^\Lambda=
\times_{\lambda\in\Lambda}\,\sA$ is the Cartesian product of $\Lambda$
copies of~$\sA$; so the objects of $\sE$ are all the $\Lambda$\+indexed
collections $E=(A_\lambda\in\sA)_{\lambda\in\Lambda}$ of objects
of~$\sA$.

 Furthermore, we are interested in the direct product functor
$G=\prod_{\lambda\in\Lambda}\:\sE\rarrow\sA$;
so $G((A_\lambda)_{\lambda\in\Lambda})=\prod_{\lambda\in\Lambda}
A_\lambda\in\sA$.
 Obviously, there are enough injective objects in $\sE$; the injective
objects of $\sE$ are precisely all the $\Lambda$\+indexed collections
of injective objects in~$\sA$.
 So the derived functors $\boR^*G$ can be constructed using injective 
coresolutions in~$\sE$.
 We will denote them by
$\prod_{\lambda\in\Lambda}^{(m)}((A_\lambda)_{\lambda\in\Lambda})
=\boR^mG((A_\lambda)_{\lambda\in\Lambda})$, \,$m\ge0$.

 More specifically, let $X$ be a scheme and $\sA=X\Qcoh$ be
the category of quasi-coherent sheaves on~$X$.
 It is known~\cite[Proposition Tag~077P]{SP} that $X\Qcoh$ is
a Grothendieck abelian category; so it has enough injective objects
by~\cite[Th\'eor\`eme~1.10.1]{GrToh}.

 Given a quasi-compact quasi-separated morphism of schemes
$f\:Y\rarrow X$, we denote by $f_*\:Y\Qcoh\rarrow X\Qcoh$ the functor
of direct image of quasi-coherent sheaves with respect to~$f$.
 The functor~$f_*$ is right adjoint to the functor of inverse image
$f^*\:X\Qcoh\rarrow Y\Qcoh$.

\begin{lem} \label{flat-affine-direct-image-lemma}
 Let $f\:Y\rarrow X$ be a flat affine morphism of schemes.
 Then, for any family of quasi-coherent sheaves
$(\N_\lambda\in Y\Qcoh)_{\lambda\in\Lambda}$ on $Y$ there are
natural isomorphisms
$$
 \prod\nolimits_{\lambda\in\Lambda}^{(m)}f_*\N_\lambda\simeq
 f_*\prod\nolimits_{\lambda\in\Lambda}^{(m)}\N_\lambda, \qquad m\ge0
$$
of quasi-coherent sheaves on~$X$.
 In particular, if\/
$\prod_{\lambda\in\Lambda}^{(m)}\N_\lambda=0$ for all $m\ge1$, then
also\/ $\prod_{\lambda\in\Lambda}^{(m)}f_*\N_\lambda=0$ for all $m\ge1$.
\end{lem}

\begin{proof}
 The assertion follows from the construction of the derived functors
$\prod_{\lambda\in\Lambda}^{(m)}$ in $Y\Qcoh$ and $X\Qcoh$ in view
of the following three observations:
\begin{itemize}
\item the direct image functor~$f_*$ preserves infinite products
(being a right adjoint functor);
\item the direct image functor~$f_*$ takes injective objects of
$Y\Qcoh$ to injective objects of $X\Qcoh$ (because $f_*$~is right
adjoint to the inverse image functor~$f^*$, which is exact for
flat morphisms~$f$);
\item the direct image functor~$f_*$ is exact (since the morphism~$f$
is affine).
\end{itemize}
 It suffices to choose an injective coresolution $0\rarrow\N_\lambda
\rarrow\J_\lambda^\bu$ in $Y\Qcoh$ for every quasi-coherent sheaf
$\N_\lambda$, notice that $0\rarrow f_*\N_\lambda \rarrow
f_*\J_\lambda^\bu$ is an injective coresolution of $f_*\N_\lambda$
in $X\Qcoh$, and compute that
\begin{multline*}
 \prod\nolimits_{\lambda\in\Lambda}^{(m)}f_*\N_\lambda=
 \cH^m\left(\prod\nolimits_{\lambda\in\Lambda}f_*\J_\lambda^\bu\right)
 \simeq
 \cH^m\left(f_*\prod\nolimits_{\lambda\in\Lambda}\J_\lambda^\bu\right)
 \\ \simeq
 f_*\cH^m\left(\prod\nolimits_{\lambda\in\Lambda}\J_\lambda^\bu\right)
 =f_*\prod\nolimits_{\lambda\in\Lambda}^{(m)}\N_\lambda,
\end{multline*}
where $\cH^m$ denotes the cohomology sheaves of a complex of
quasi-coherent sheaves.
\end{proof}

\subsection{\v Cech coresolution}
 Let $X$ be a quasi-compact quasi-separated scheme and
$X=\bigcup_{\alpha=1}^N U_\alpha$ be its finite covering by
quasi-compact open subschemes $U_\alpha\subset X$.
 For any subsequence of indices $1\le\alpha_1<\dotsb<\alpha_r\le N$,
put $U_{\alpha_1,\dotsc,\alpha_r}=\bigcap_{s=1}^rU_{\alpha_s}
\subset X$; then $U_{\alpha_1,\dotsc,\alpha_r}$ is also
a quasi-compact open subscheme in~$X$.
 Denote its open immersion morphism by $j_{\alpha_1,\dotsc,\alpha_r}\:
U_{\alpha_1,\dotsc,\alpha_r}\rarrow X$ (in particular, the open
immersion morphisms $U_\alpha\rarrow X$ are denoted by~$j_\alpha$).

\begin{lem} \label{cech-coresolution-lemma}
 For any quasi-coherent sheaf\/ $\M$ on $X$, there is a natural
exact sequence of quasi-coherent sheaves on~$X$
\begin{multline} \label{cech-coresolution}
 0\lrarrow\M\lrarrow\bigoplus\nolimits_{\alpha=1}^N
 j_\alpha{}_*j_\alpha^*\M
 \lrarrow\bigoplus\nolimits_{1\le\alpha<\beta\le N}
 j_{\alpha,\beta}{}_*j_{\alpha,\beta}^*\M \\ \lrarrow\dotsb\lrarrow
 j_{1,\dotsc,N}{}_*j_{1,\dotsc,N}^*\M\lrarrow0.
\end{multline}
\end{lem}

\begin{proof}
 The construction of the complex~\eqref{cech-coresolution} is
straightforward.
 In order to show that it is an exact sequence, one can check that
its restriction to every $U_\alpha$ (i.~e., the inverse image with
respect to~$j_\alpha$) is contractible.
 (Cf.~\cite[Lemma~III.4.2]{HarAG}.)
\end{proof}

\subsection{Roos axiom holds for quasi-compact semi-separated schemes}
 Now we can spell out the proof of the theorem of
Hogadi--Xu~\cite[Remark~3.3]{HX} and
Herbera--Pitsch--Saor\'\i n--Virili~\cite[Theorem~8.27]{HPSV}.

\begin{thm} \label{roos-axiom-cech-coresolution-semi-separated}
 Let $X=\bigcup_{\alpha=1}^N U_\alpha$ be a quasi-compact semi-separated
scheme covered by $N$ affine open subschemes.
 Then the Roos axiom $\AB4^*$\+$n$ from\/~\cite[Definition~1.1]{Roos}
with the parameter $n=N-1$ holds for the Grothendieck category $X\Qcoh$.
 In other words, for any set\/ $\Lambda$ and any family of
quasi-coherent sheaves $(\M_\lambda\in X\Qcoh)_{\lambda\in\Lambda}$,
one has\/ $\prod_{\lambda\in\Lambda}^{(i)}\M_\lambda=0$ for all $i>n$.
\end{thm}

\begin{proof}
 For every $1\le r\le N$, denote by $U_r$ the disjoint union
$$
 U_r=\coprod\nolimits_{1\le\alpha_1<\dotsb<\alpha_r\le N}
 U_{\alpha_1}\cap\dotsb\cap U_{\alpha_r}.
$$
 This is a finite disjoint union of affine schemes, so $U_r$ is
an affine scheme.
 Clearly, the natural morphism $f_r\:U_r\rarrow X$ is flat and affine.

 Let $\Lambda$ be a set and $(\M_\lambda\in X\Qcoh)_{\lambda\in\Lambda}$
be a family of quasi-coherent sheaves on~$X$.
 By Lemma~\ref{cech-coresolution-lemma}, each sheaf $\M_\lambda$ has
the \v Cech coresolution~\eqref{cech-coresolution}, which has the form
\begin{equation} \label{cech-coresolution-rewritten}
 0\lrarrow\M_\lambda\lrarrow f_1{}_*\N_{\lambda,1}\lrarrow\dotsb
 \lrarrow f_N{}_*\N_{\lambda,N}\lrarrow0
\end{equation}
for suitable quasi-coherent sheaves $\N_{\lambda,r}$ on~$U_r$.
 The scheme $U_r$ is affine, so the direct product functors in
$U_r\Qcoh$ are exact and
$\prod^{(m)}_{\lambda\in\Lambda}\N_{\lambda,r}=0$ for all $m\ge1$.

 By Lemma~\ref{flat-affine-direct-image-lemma}, it follows that
$\prod^{(m)}_{\lambda\in\Lambda}f_r{}_*\N_{\lambda,r}=0$ for
all $m\ge1$ and $1\le r\le N$.
 Finally, the \v Cech coresolution~\eqref{cech-coresolution-rewritten}
has length $N-1$, so Lemma~\ref{acyclic-models} implies that
$\prod^{(i)}_{\lambda\in\Lambda}\M_\lambda=0$ for all $i>N-1$.
\end{proof}

\Section{\v Cech Coresolution Argument for Noetherian Schemes}
\label{cech-noetherian-secn}

 The aim of this section is to prove the Roos axiom for the categories
$X\Qcoh$ of quasi-coherent sheaves on Noetherian schemes of finite
Krull dimension.

\subsection{Weakly adjusted objects}
 As in Section~\ref{adjusted-objects-subsecn}, we consider an abelian
category $\sE$ with enough injective objects, an abelian category $\sA$,
and a left exact functor $G\:\sE\rarrow\sA$.

\begin{lem} \label{weakly-adjusted-objects-lemma}
 Let $E\in\sE$ be an object, and let\/ $0\rarrow E\rarrow K^0\rarrow
K^1\rarrow\dotsb\rarrow K^d\rarrow0$ be a finite coresolution of $E$ by
objects $K^i\in\sE$ such that\/ $\boR^mG(K^i)=0$ for all $i\ge0$ and
$m>d-i$.
 Then one has\/ $\boR^iG(E)=0$ for all $i>d$.
\end{lem}

\begin{proof}
 The case $d=0$ is obvious.
 For $d\ge1$, one can argue by induction on~$d$.
 Denoting the image of the morphism $K^0\rarrow K^1$ by $F$, we have
$\boR^iG(F)=0$ for all $i>d-1$ by the induction assumption, and
it remains to use the long exact sequence $\dotsb\rarrow\boR^{i-1}G(F)
\rarrow\boR^iG(E)\rarrow\boR^iG(K^0)\rarrow\dotsb$.
\end{proof}

\subsection{Adjusted classes} \label{adjusted-classes-subsecn}
 Let $\sE$ be an abelian category with enough injective objects.
 We will say that a class of objects $\sK\subset\sE$ is
\emph{coresolving} if it satisfies the following conditions:
\begin{enumerate}
\renewcommand{\theenumi}{\roman{enumi}}
\item the full subcategory $\sK$ is closed under extensions in~$\sE$;
\item the full subcategory $\sK$ is closed under the cokernels of
monomorphisms in~$\sE$;
\item all the injective objects of $\sE$ belong to~$\sK$.
\end{enumerate}

 Let $\sE$ be an abelian category and $\sK\subset\sE$ be a coresolving
class of objects in~$\sE$.
 One says that the \emph{coresolution dimension} of an object
$E\in\sE$ with respect to the coresolving class $\sK\subset\sE$
does not exceed an integer~$d\ge0$ if there exists an exact sequence
$0\rarrow E\rarrow K^0\rarrow K^1\rarrow\dotsb\rarrow K^d\rarrow0$
in $\sE$ with $K^i\in\sK$ for all $0\le i\le d$.
 One can show that the $\sK$\+coresolution dimension of an object
$E\in\sE$ does not depend on the choice of a coresolution of $E$
by objects from $\sK$, in an obvious sense~\cite[Lemma~2.1]{Zhu},
\cite[Proposition~2.3]{Sto}, \cite[Corollary~A.5.2]{Pcosh}.

 Let $\sA$ be another abelian category and $G\:\sE\rarrow\sA$ be
a left exact functor.
 We will say that a coresolving class of objects $\sK\subset\sE$ is
\emph{adjusted to~$G$} if it satisfies the following condition:
\begin{enumerate}
\renewcommand{\theenumi}{\roman{enumi}}
\setcounter{enumi}{3}
\item for any short exact sequence $0\rarrow K'\rarrow K\rarrow K''
\rarrow0$ in $\sE$ with the objects $K'$, $K$, $K''$ belonging to $\sK$,
the short sequence
$$
 0\lrarrow G(K')\lrarrow G(K)\lrarrow G(K'')\lrarrow0
$$
is exact in~$\sA$.
\end{enumerate}

\begin{lem} \label{adjusted-class-lemma}
 Let\/ $\sE$ be an abelian category with enough injective objects,
$\sA$ be an abelian category, $G\:\sE\rarrow\sA$ be a left exact
functor, and\/ $\sK\subset\sE$ be a coresolving class of objects
adjusted to~$G$.
 Let $E\in\sE$ be an object and\/ $0\rarrow E\rarrow K^0\rarrow K^1
\rarrow K^2\rarrow\dotsb$ be a coresolution of the object $E$ by
objects $K^i\in\sK$ in the abelian category\/~$\sE$.
 Then there are natural isomorphisms
$$
 \boR^iG(E)\simeq H^iG(K^\bu), \qquad i\ge0
$$
of objects in\/~$\sA$.
\end{lem}

\begin{proof}
 This lemma is very well known.
 Given an object $K\in\sK$, consider its injective coresolution
$0\rarrow K\rarrow J^0\rarrow J^1\rarrow J^2\rarrow\dotsb$ in~$\sE$.
 Put $Z^0=K$, and denote by $Z^i$ the image of the morphism
$J^{i-1}\rarrow J^i$ for $i\ge1$.
 So the long exact sequence $0\rarrow K\rarrow J^\bu$ can be obtained
by splicing the short exact sequences $0\rarrow Z^i\rarrow J^i
\rarrow Z^{i+1}\rarrow 0$ in~$\sE$.
 Then, by~(iii), we have $J^i\in\sK$ for all $i\ge0$, and by~(ii),
it follows by induction that $Z^i\in\sK$ for all $i\ge0$.
 By~(iv), the short sequences $0\rarrow G(Z^i)\rarrow G(J^i)\rarrow
G(Z^{i+1})\rarrow0$ are exact in $\sA$ for all $i\ge0$.
 So the complex $0\rarrow G(K)\rarrow G(J^\bu)$ is exact in $\sA$,
and we have shown that $\boR^mG(K)=0$ for all $m>0$.
 Returning to the situation at hand, we have $\boR^mG(K^i)=0$
for all $i\ge0$ and $m>0$, and it remains to apply
Lemma~\ref{acyclic-models}.
\end{proof}

\subsection{Flasque quasi-coherent sheaves}
 Let $X$ be a topological space and $\Q$ be a sheaf of abelian groups
on~$X$.
 One says that the sheaf $\Q$ is \emph{flasque} if the restriction map
$\Q(U)\rarrow\Q(V)$ is surjective for every pair of open subsets
$V\subset U\subset X$.
 A quasi-coherent sheaf $\Q$ on a scheme $X$ is said to be flasque if
it is flasque as a sheaf of abelian groups.

\begin{lem}
 Let $X$ be a (locally) Noetherian scheme.
 Then the class of all flasque quasi-coherent sheaves is
coresolving in $X\Qcoh$.
\end{lem}

\begin{proof}
 The class of all flasque sheaves of abelian groups is closed under
extensions and cokernels of monomorphisms in the abelian category of
sheaves of abelian groups on any topological space $X$
by~\cite[Th\'eor\`eme and Corollaire~II.3.1.2]{God}.
 On any ringed space $(X,\cO_X)$, all injective sheaves of
$\cO_X$\+modules are flasque~\cite[Lemma~III.2.4]{HarAG}.
 On a locally Noetherian scheme $X$, all injective quasi-coherent
sheaves are injective as sheaves of
$\cO_X$\+modules~\cite[Theorem~II.7.18]{HarRD},
\cite[Theorem~A.3]{EP}; hence they are also flasque.
\end{proof}

\begin{lem} \label{flasqueness-preservation-by-direct-inverse-images}
 Let $f\:Y\rarrow X$ be a morphism of schemes.  Then \par
\textup{(a)} if the morphism~$f$ is quasi-compact and quasi-separated,
then the direct image functor $f_*\:Y\Qcoh\rarrow X\Qcoh$ preserves
flasqueness of quasi-coherent sheaves; \par
\textup{(b)} if the morphism~$f$ is an open immersion, then
the inverse image functor $f^*\:X\Qcoh\rarrow Y\Qcoh$ preserves
flasqueness of quasi-coherent sheaves.
\end{lem}

\begin{proof}
 Both assertions are obvious.
 The point is that the direct images and restrictions to open subsets
preserve flasqueness of sheaves of abelian groups.
\end{proof}

\begin{lem} \label{flasque-adjusted-to-direct-images}
 Let $f\:Y\rarrow X$ be a morphism of Noetherian schemes.
 Then the coresolving class of flasque quasi-coherent sheaves on $Y$
is adjusted to the left exact functor of direct image
$f_*\:Y\Qcoh\rarrow X\Qcoh$.
\end{lem}

\begin{proof}
 It suffices to observe that the coresolving class of flasque sheaves
of abelian groups on any topological space is adjusted to the functor
of sections over every open subset~\cite[Th\'eor\`eme~II.3.1.2]{God}.
\end{proof}

\begin{lem} \label{flasque-dimension-Krull-dimension}
 Let $X$ be a Noetherian scheme of finite Krull dimension~$D$.
 Then the coresolution dimension of any quasi-coherent sheaf on $X$
with respect to the coresolving subcategory of flasque quasi-coherent
sheaves does not exceed~$D$.
\end{lem}

\begin{proof}
 The point is that, on any Noetherian topological space of
combinatorial dimension~$D$, the coresolution dimension of any
sheaf of abelian groups with respect to the coresolving subcategory
of flasque sheaves of abelian groups does not exceed~$D$.
 This is a version of the Grothendieck vanishing
theorem~\cite[Th\'eor\`eme~3.6.5]{GrToh};
see~\cite[Lemma~3.4.7(a)]{Pcosh}.
\end{proof}

\subsection{Derived functors of direct image}
 Given a quasi-compact quasi-separated morphism of schemes
$f\:Y\rarrow X$, we denote by $\boR^*f_*$ the right derived functors
of the left exact functor of direct image $f_*\:Y\Qcoh\rarrow X\Qcoh$
(constructed, as usual, in terms of injective coresolutions).
 Similarly, we denote by $\boR^*\Gamma(X,{-})$ the right derived
functors of the functor of global sections $\Gamma(X,{-})\:X\Qcoh
\rarrow\sAb$ taking values in the category of abelian groups~$\sAb$.

\begin{lem} \label{flat-affine-direct-image-global-sections}
 Let $f\:Y\rarrow X$ be a flat affine morphism of schemes.
 Then, for every quasi-coherent sheaf\/ $\N$ on $Y$, there are
natural isomorphisms of abelian groups
$$
 \boR^m\Gamma(X,f_*\N)\simeq\boR^m\Gamma(Y,\N).
$$
 In particular, if\/ $\boR^m\Gamma(Y,\N)=0$ for all $m\ge1$, then
also\/ $\boR^m\Gamma(X,f_*\N)=0$ for all $m\ge1$.
\end{lem}

\begin{proof}
 This is similar to (but simpler than)
Lemma~\ref{flat-affine-direct-image-lemma}.
 One uses the observations that there is a natural isomorphism
$\Gamma(X,f_*\N)\simeq\Gamma(Y,\N)$, and the functor~$f_*$ is exact
and takes injective objects to injective objects.
\end{proof}

\begin{lem} \label{homological-dimension-of-global-sections}
 Let $X$ be a quasi-compact semi-separated scheme with an affine
open covering $X=\bigcup_{\alpha=1}^N U_\alpha$.
 Then for every quasi-coherent sheaf $\M$ on $X$ one has\/
$\boR\Gamma^i(X,\M)=0$ for all $i>N-1$.
\end{lem}

\begin{proof}
 This is very well known.
 We use the notation from the proof of
Theorem~\ref{roos-axiom-cech-coresolution-semi-separated}.
 By Lemma~\ref{cech-coresolution-lemma}, the quasi-coherent sheaf $\M$
has the \v Cech coresolution~\eqref{cech-coresolution} of the form
$0\rarrow\M\rarrow f_1{}_*\N_1\rarrow\dotsb\rarrow f_N{}_*\N_N\rarrow0$
for suitable quasi-coherent sheaves $\N_r$ on the schemes~$U_r$.
 The functors of global sections $\Gamma(U_r,{-})$ on affine schemes
$U_r$ are exact, so by
Lemma~\ref{flat-affine-direct-image-global-sections} we have
$\boR^m\Gamma(X,f_r{}_*\N_r)=0$ for all $m\ge1$ and $1\le r\le N$.
 It remains to refer to Lemma~\ref{acyclic-models}.
\end{proof}

\begin{lem} \label{homological-dimension-of-direct-image}
 Let $f\:Y\rarrow X$ be a morphism of schemes and
$X=\bigcup_\alpha U_\alpha$ be an affine open covering of~$X$.
 Assume that the scheme $Y$ is semi-separated and Noetherian and,
for every index~$\alpha$, the scheme $f^{-1}(U_\alpha)=
U_\alpha\times_XY$ can be covered by $d+1$ affine open subschemes
(where $d\ge0$ is an integer).
 Then, for every quasi-coherent sheaf $\N$ on $Y$, one has\/
$\boR^kf_*(\N)=0$ for all $k>d$.
\end{lem}

\begin{proof}
 The point is that the restrictions of injective quasi-coherent sheaves
on a Noetherian scheme $Y$ to its open subschemes $V\subset Y$ are
injective quasi-coherent sheaves on~$V$ \,\cite[Theorem~II.7.18]{HarRD},
\cite[Theorem~A.3]{EP}.
 Consequently, the natural isomorphism
$$
 \Gamma(U,f_*(\N)|_U)\simeq\Gamma(f^{-1}(U),\N|_{f^{-1}(U)})
$$
implies
$$
 \Gamma(U,\boR^mf_*(\N)|_U)\simeq
 \boR^m\Gamma(f^{-1}(U),\N|_{f^{-1}(U)})
$$
for all $m\ge0$, \ $\N\in Y\Qcoh$, and affine open subschemes
$U\subset X$.
 It remains to take $U=U_\alpha$ and apply
Lemma~\ref{homological-dimension-of-global-sections} to the scheme
$f^{-1}(U_\alpha)$, the sheaf $\M=\N|_{f^{-1}(U_\alpha)}$, and
the integer $N=d+1$.
\end{proof}

\subsection{Direct images under flat morphisms}
 The following lemma is a weak version of
Lemma~\ref{flat-affine-direct-image-lemma} for nonaffine morphisms~$f$.

\begin{lem} \label{flat-nonaffine-direct-image-lemma}
 Let $f\:Y\rarrow X$ be a flat morphism of schemes and $t$, $d\ge0$
be two integers.
 Let\/ $\Lambda$ be a set and
$(\Q_\lambda\in Y\Qcoh)_{\lambda\in\Lambda}$ be a family of
quasi-coherent sheaves on $Y$ such that\/ $\boR^mf_*\Q_\lambda=0$
for all $\lambda\in\Lambda$ and $m\ge1$.
 Assume that\/
$\prod_{\lambda\in\Lambda}^{(l)}\Q_\lambda=0$ in $Y\Qcoh$ for all $l>t$. 
 Suppose further that\/ $\boR^kf_*\M=0$ for all quasi-coherent
sheaves $\M$ on $Y$ and all $k>d$.
 Then\/ $\prod_{\lambda\in\Lambda}^{(i)}f_*\Q_\lambda=0$ in $X\Qcoh$
for all $i>t+d$.
\end{lem}

\begin{proof}
 Arguing similarly to the proof of
Lemma~\ref{flat-affine-direct-image-lemma}, we observe that
the functor $f_*\:Y\Qcoh\rarrow X\Qcoh$ preserves infinite products
and takes injective objects to injective objects.
 For any chosen injective coresolutions $0\rarrow\Q_\lambda
\rarrow\J_\lambda^\bu$ in $Y\Qcoh$, the complexes $0\rarrow
f_*\Q_\lambda\rarrow f_*\J_\lambda^\bu$ are exact in $X\Qcoh$ by
assumption; so they are injective coresolutions in $X\Qcoh$.
 Consequently, we have
$$
 \prod\nolimits_{\lambda\in\Lambda}^{(m)}f_*\Q_\lambda=
 \cH^m\left(\prod\nolimits_{\lambda\in\Lambda}f_*\J_\lambda^\bu\right)
 \simeq
 \cH^m\left(f_*\prod\nolimits_{\lambda\in\Lambda}\J_\lambda^\bu\right),
 \qquad m\ge0.
$$

 Put $\I^\bu=\prod_{\lambda\in\Lambda}\J_\lambda^\bu$.
 By assumption, we have $\cH^l(\I^\bu)=0$ for all $l>t$.
 Denote by $\M$ the kernel of the morphism $\I^t\rarrow\I^{t+1}$
in $Y\Qcoh$.
 Then the complex $0\rarrow\M\rarrow\I^t\rarrow\I^{t+1}\rarrow
\I^{t+2}\rarrow\dotsb$ is exact, so $\I^{t+\bu}$ is an injective
coresolution of $\M$ in $Y\Qcoh$ and $\boR^k f_*\M=
\cH^{t+k}(f_*\I^\bu)$ for all $k\ge1$.
 By assumption, it follows that $\cH^i(f_*\I^\bu)=0$ for $i>t+d$.
\end{proof}

\subsection{Roos axiom holds for finite-dimensional Noetherian schemes}
 Let $X$ be a Noetherian scheme with an affine open covering
$X=\bigcup_{\alpha=1}^N U_\alpha$.
 For every subsequence of indices $1\le\alpha_1<\dotsb<\alpha_r\le N$,
denote by $N_{r;\,\alpha_1,\dotsc,\alpha_r}$ the minimal cardinality
of an affine open covering of the intersection
$U_{\alpha_1}\cap\dotsb\cap U_{\alpha_r}$.
 For every $1\le r\le N$, denote by $N_r$ the maximum of all the numbers
$N_{r;\,\alpha_1,\dotsc,\alpha_r}$ (where $r$~is fixed, while
$\alpha_1$,~\dots, $\alpha_r$ vary).
 Put $N_{N+1}=0$.
 Finally, consider the number
\begin{equation} \label{combinatorial-number-M}
 M = \max\nolimits_{r=1}^N(r+N_r+\max(N_r,N_{r+1})-3),
\end{equation}
where $\max$ denotes the maximal value in a finite set of integers.

\begin{prop} \label{roos-axiom-cech-coresolution-flasque}
 Let $X$ be a Noetherian scheme and $M$ be the combinatorial
number~\eqref{combinatorial-number-M}.
 Then, for any set\/ $\Lambda$ and any family of \emph{flasque}
quasi-coherent sheaves $(\Q_\lambda)_{\lambda\in\Lambda}$ on $X$,
one has\/ $\prod_{\lambda\in\Lambda}^{(i)}\Q_\lambda=0$ in $X\Qcoh$
for all $i>M$.
\end{prop}

\begin{proof}
 We use the notation $U_r=\coprod_{1\le\alpha_1<\dotsb<\alpha_r\le N}
U_{\alpha_1}\cap\dotsb\cap U_{\alpha_r}$ and $f_r\:U_r\rarrow X$
from the proof of
Theorem~\ref{roos-axiom-cech-coresolution-semi-separated}.
 The schemes $U_r$ are not necessarily affine now, but they are
semi-separated (as open subschemes in affine schemes).
 The morphism~$f_r$ is no longer affine, either; but it is still flat.
 Furthermore, for every $1\le r\le N$ the scheme $U_r$ can be covered
by $N_r$ affine open subschemes (since $U_r$ is a finite disjoint union
of schemes that can be so covered).

 In addition, we observe that, for any index~$\alpha$ and
any subsequence of indices $1\le\alpha_1<\dotsb<\alpha_r\le N$,
the scheme $U_\alpha\cap(U_{\alpha_1}\cap\dotsb\cap U_{\alpha_r})$
can be covered by $N_r$ affine open subschemes if
$\alpha\in\{\alpha_1,\dotsc,\alpha_r\}$ and by $N_{r+1}$ affine
open subschemes otherwise.
 So, in any case, $U_\alpha\cap(U_{\alpha_1}\cap\dotsb
\cap U_{\alpha_r})$ can be covered by $\max(N_r,N_{r+1})$ affine
open subschemes.
 It follows that the scheme $f_r^{-1}(U_\alpha)=U_\alpha\times_XU_r$
also can be covered by $\max(N_r,N_{r+1})$ affine open subschemes.

 Since the sheaves $\Q_\lambda$ are flasque on $X$ and
the morphism~$f_r$ is a disjoint union of open immersions into $X$,
Lemma~\ref{flasqueness-preservation-by-direct-inverse-images}(b)
implies that the sheaves $f_r^*\Q_\lambda$ are flasque on~$U_r$.
 By Lemmas~\ref{adjusted-class-lemma}
and~\ref{flasque-adjusted-to-direct-images}, it follows that
$\boR^mf_r{}_*(f_r^*\Q_\lambda)=0$ for all $1\le r\le N$, \
$\lambda\in\Lambda$, and $m\ge1$.
 By Theorem~\ref{roos-axiom-cech-coresolution-semi-separated}, we have
$\prod_{\lambda\in\Lambda}^{(l)}f_r^*Q_\lambda=0$ in $U_r\Qcoh$
for all $l>N_r-1$.
 Finally, by Lemma~\ref{homological-dimension-of-direct-image}, we
have $\boR^kf_r{}_*\M=0$ for every quasi-coherent sheaf $\M$ on
$U_r$ and all $k>\max(N_r,N_{r+1})-1$.
 Thus, Lemma~\ref{flat-nonaffine-direct-image-lemma} is applicable,
and we can conclude that $\prod_{\lambda\in\Lambda}^{(i)}
f_r{}_*f_r^*\Q_\lambda=0$ in $X\Qcoh$ for all
$i>N_r+\max(N_r,N_{r+1})-2$.

 It remains to recall the \v Cech
coresolutions~\eqref{cech-coresolution} from
Lemma~\ref{cech-coresolution-lemma}, which have the form
$$
 0\lrarrow\Q_\lambda\lrarrow f_1{}_*f_1^*\Q_\lambda\lrarrow
 f_2{}_*f_2^*\Q_\lambda\lrarrow\dotsb\lrarrow f_N{}_*f_N^*\Q_\lambda
 \lrarrow0.
$$
 By Lemma~\ref{weakly-adjusted-objects-lemma}, it follows that
$\prod_{\lambda\in\Lambda}^{(i)}\Q_\lambda=0$ in $X\Qcoh$ whenever
$i>r+N_r+\max(N_r,N_{r+1})-3$ for all $1\le r\le N$.
\end{proof}

\begin{thm} \label{roos-axiom-cech-coresolution-finite-krull-dim}
 Let $X$ be a Noetherian scheme of finite Krull dimension $D$ and $M$
be the combinatorial number~\eqref{combinatorial-number-M}.
 Then the Roos axiom $\AB4^*$\+$n$ with $n=M+D$ holds for
the Grothendieck category $X\Qcoh$.
 In other words, for any set\/ $\Lambda$ and any family of
quasi-coherent sheaves $(\M_\lambda)_{\lambda\in\Lambda}$ on $X$,
one has\/ $\prod_{\lambda\in\Lambda}^{(i)}\M_\lambda=0$ in $X\Qcoh$
for all $i>M+D$.
\end{thm}

\begin{proof}
 By Lemma~\ref{flasque-dimension-Krull-dimension}, the sheaves
$\M_\lambda$ have flasque coresolutions of length $D$ in $X\Qcoh$,
$$
 0\lrarrow\M_\lambda\lrarrow\Q_\lambda^0\lrarrow\Q_\lambda^1
 \lrarrow\dotsb\lrarrow\Q_\lambda^D\lrarrow0.
$$
 By Proposition~\ref{roos-axiom-cech-coresolution-flasque}, we have
$\prod_{\lambda\in\Lambda}^{(m)}\Q_\lambda^s=0$ for all $0\le s\le D$,
\ $\lambda\in\Lambda$, and $m>M$.
 By virtue of Lemma~\ref{weakly-adjusted-objects-lemma}, it follows
that $\prod_{\lambda\in\Lambda}^{(i)}\M_\lambda=0$ for all $i>M+D$.
\end{proof}

\Section{Generator of Finite Projective Dimension
for~Semi-Separated~Schemes}  \label{generator-secn}

 The aim of this section is to offer an alternative proof of
the Roos axiom $\AB4^*$\+$n$ for the category of quasi-coherent
sheaves $X\Qcoh$ on a quasi-compact semi-separated scheme $X$
covered by $N$ affine open subschemes.
 It gives a slightly weaker numerical bound $n=N$ for the parameter~$n$
in the Roos axiom as compared to the \v Cech coresolution argument in
Section~\ref{cech-semi-separated-secn}, but provides a conceptually
stronger result that the category $X\Qcoh$ has a generator of
finite projective dimension~$N$.

\subsection{Contraadjusted and very flat modules}
\label{cta-vfl-modules-subsecn}
 Let $R$ be a commutative ring.
 Given an element $s\in R$, we consider the localization $R[s^{-1}]=
S^{-1}R$ of the ring $R$ at the multiplicative subset $S=\{1,s,s^2,
s^3,\dotsc\}$.
 So $R[s^{-1}]$ is the ring of functions on the principal affine open
subscheme in $\Spec R$ corresponding to the element~$s$.
 An $R$\+module $C$ is said to be
\emph{contraadjusted}~\cite[Section~1.1]{Pcosh}
if $\Ext^1_R(R[s^{-1}],C)=0$ for every $s\in R$.
 An $R$\+module $F$ is said to be \emph{very flat} if
$\Ext^1_R(F,C)=0$ for every contraadjusted $R$\+module~$C$.

 The word ``contraadjusted'' is supposed to mean ``adjusted for
contraherent cosheaves''.
 So contraadjusted modules were introduced in~\cite{Pcosh} as a basic
building block of the contraherent cosheaf theory.
 The definitions of contraadjusted and very flat modules given above
fall under the general scheme of a concept known as a \emph{cotorsion
pair} (or \emph{cotorsion theory}) of classes of modules, or more
generally, classes of objects in an abelian or exact category.
 The definition of a cotorsion theory is due to Salce~\cite{Sal},
and the main result in this connection is the \emph{Eklof--Trlifaj
theorem}~\cite[Theorems~2 and~10]{ET}.
 The main reference source for cotorsion pairs in module categories
is the book~\cite{GT}.
 An introductory discussion tailored to the situation at hand can be
found in~\cite[Section~4]{Pphil};
see specifically~\cite[Theorem~4.5]{Pphil}.

 In particular, the pair of classes (very flat $R$\+modules,
contraadjusted $R$\+modules) is a complete cotorsion pair in
the abelian category of $R$\+modules $R\Modl$ (generated by the set of
$R$\+modules $R[s^{-1}]$, \,$s\in R$) \cite[Theorem~1.1.1]{Pcosh}.
 It is called the \emph{very flat cotorsion pair} in $R\Modl$.
 As one of the aspects of the Eklof--Trlifaj theorem, one obtains
the following description of very flat $R$\+modules: an $R$\+module
is very flat if and only if it is a direct summand of a module
admitting a transfinite (ordinal-indexed) increasing filtration whose
successive quotients are $R$\+modules of the form $R[s^{-1}]$,
\,$s\in R$ \,\cite[Corollary~1.1.4]{Pcosh}.
 Any projective $R$\+module is very flat, and any very flat $R$\+module
is flat.

 In addition, one observes that the $R$\+module $R[s^{-1}]$ has
projective dimension at most~$1$ \,\cite[Section~1.1]{Pcosh},
\cite[proof of Lemma~2.1]{Pcta}.
 It follows that any quotient module of a contraadjusted $R$\+module
is contraadjusted, and that any very flat $R$\+module has projective
dimension at most~$1$.
 Furthermore, the very flat cotorsion pair is \emph{hereditary}:
the kernel of any surjective morphism of very flat $R$\+modules is
very flat.

\begin{ex} \label{open-subscheme-very-flat}
 Let $X$ be an affine scheme and $U\subset X$ be an affine open
subscheme in~$X$.
 Then the ring of functions $S=\cO(U)$ is a very flat module over
the ring $R=\cO(X)$.
 This is~\cite[Lemma~1.2.4]{Pcosh}.
\end{ex}

 Given a commutative ring homomorphism $f\:R\rarrow S$, one says that
$f$~is \emph{very flat} (or equivalently, $S$ is a \emph{very flat
$R$\+algebra}) if, for every element $s\in S$, the $R$\+module
$S[s^{-1}]$ is very flat.
 Notice that the condition that the $R$\+module $S$ itself be very
flat is strictly weaker and \emph{not}
sufficient~\cite[Section~6.1]{Pphil}.

 The \emph{Very Flat Conjecture} (one of possible formulations)
stated that, for any commutative ring $R$, any finitely presented
flat commutative $R$\+algebra is very flat.
 This was proved in~\cite[Main Theorem~1.1 and Corollary~9.1]{PSl1}.
 We will not use this important and difficult result in this paper.

 The main construction below in Section~\ref{very-flat-sheaves-subsecn}
is based on the properties of very flat modules listed in the following
lemmas.

\begin{lem} \label{ascent-lemma}
 Let $R\rarrow S$ be a homomorphism of commutative rings.
 Then, for any very flat $R$\+module $F$, the $S$\+module $S\ot_RF$
is very flat.
\end{lem}

\begin{proof}
 This is~\cite[Lemma~1.2.2(b)]{Pcosh}.
\end{proof}

\begin{lem} \label{direct-image-lemma}
 Let $R\rarrow S$ be a very flat homomorphism of commutative rings.
 Then any very flat $S$\+module $G$ is very flat as an $R$\+module.
\end{lem}

\begin{proof}
 This is~\cite[Lemma~1.2.3(b)]{Pcosh}.
\end{proof}

\begin{lem} \label{descent-lemma}
 Let $f_\alpha\:R\rarrow S_\alpha$ be a finite family of homomorphisms
of commutative rings such that the related family of morphisms of
the spectra\/ $\Spec S_\alpha\rarrow\Spec R$ is an open covering
of\/ $\Spec R$ (in particular, $\Spec S_\alpha\rarrow\Spec R$ is
the embedding of an affine open subscheme for every~$\alpha$).
 Then an $R$\+module $F$ is very flat \emph{if and only if}
the $S_\alpha$\+module $S_\alpha\ot_RF$ is very flat for every~$\alpha$.
\end{lem}

\begin{proof}
 This is~\cite[Lemma~1.2.6(a)]{Pcosh}.
\end{proof}

 A further discussion of Lemmas~\ref{ascent-lemma}\+-\ref{descent-lemma}
can be found in the paper~\cite[Example~2.5]{Pal}.

\subsection{Very flat quasi-coherent sheaves}
\label{very-flat-sheaves-subsecn}
 Let $X$ be a scheme.
 A quasi-coherent sheaf $\F$ on $X$ is said to be \emph{very
flat}~\cite[Section~1.10]{Pcosh} if, for every affine open subscheme
$U\subset X$, the $\cO_X(U)$\+module $\F(U)$ is very flat.
 In view of Lemmas~\ref{ascent-lemma} and~\ref{descent-lemma}, it
suffices to check this condition for affine open subschemes
$U=U_\alpha$ belonging to any given affine open covering
$X=\bigcup_\alpha U_\alpha$ of the scheme~$X$.
 In particular, if $X=V\cup W$ is a (not necessarily affine) open
covering of $X$, then a quasi-coherent sheaf $\F$ on $X$ is very flat
if and only if the quasi-coherent sheaf $\F|_V$ is very flat on $V$
and the quasi-coherent sheaf $\F|_W$ is very flat on~$W$.

 The class of all very flat quasi-coherent sheaves on a scheme $X$
is closed under extensions, kernels of surjective morphisms, and
infinite direct sums (because the class of all very flat modules
over a commutative ring has these closure properties).

 For any morphism of schemes $f\:Y\rarrow X$, the inverse image
functor $f^*\:X\Qcoh\allowbreak\rarrow Y\Qcoh$ takes very flat
quasi-coherent sheaves on $X$ to very flat quasi-coherent sheaves on~$Y$
(by Lemma~\ref{ascent-lemma}).

 A morphism of schemes $f\:Y\rarrow X$ is said to be \emph{very flat}
if, for any pair of affine open subschemes $U\subset X$ and
$V\subset Y$ such that $f(V)\subset U$, the $\cO_X(U)$\+module
$\cO_Y(V)$ is very flat~\cite[Section~1.10]{Pcosh}.
 A morphism of affine schemes $\Spec S\rarrow\Spec R$ is very flat if
and only if the related morphism of commutative rings $R\rarrow S$ is
very flat (in the sense of the definition in
Section~\ref{cta-vfl-modules-subsecn}).

 For any very flat affine morphism of schemes $f\:Y\rarrow X$,
the direct image functor~$f_*$ takes very flat quasi-coherent sheaves
on $Y$ to very flat quasi-coherent sheaves on~$X$
(by Lemma~\ref{direct-image-lemma}).
 In particular, for any open subscheme $W$ in a scheme $X$, the open
embedding morphism $j\:W\rarrow X$ is very flat (by
Example~\ref{open-subscheme-very-flat}).
 Therefore, if the morphism~$j$ is affine, then the quasi-coherent
sheaf $j_*\G$ on $X$ is very flat for any very flat quasi-coherent
sheaf $\G$ on~$W$.

 The following lemma is a very flat version of~\cite[Lemma~A.1]{EP}
and a part of~\cite[Lemma~4.1.1]{Pcosh}.

\begin{lem} \label{very-flat-qcoh}
 On any quasi-compact semi-separated scheme $X$, any quasi-coherent
sheaf is a quotient sheaf of a very flat quasi-coherent sheaf.
\end{lem}

\begin{proof}
 Let $\M$ be a quasi-coherent sheaf on~$X$.
 Arguing by induction, suppose that we have already constructed
a surjective morphism of quasi-coherent sheaves $\cH\rarrow\M$ on
$X$ such that, for a certain open subscheme $V\subset X$,
the restriction $\cH|_V$ is a very flat quasi-coherent sheaf on~$V$.
 Let $U\subset X$ be an affine open subscheme in $X$; denote its open
embedding morphism by $j\:U\rarrow X$.
 Since $X$ is a semi-separated scheme by assumption and $U$ is
an affine scheme, the morphism~$j$ is affine.

 Consider the restriction/inverse image $\cH|_U=j^*\cH$ on~$U$.
 Any $\cO(U)$\+module is a quotient module of a very flat
$\cO(U)$\+module; so we have a short exact sequence of
quasi-coherent sheaves $0\rarrow\K\rarrow\G\rarrow j^*\cH\rarrow0$
on $U$ with a very flat quasi-coherent sheaf~$\G$.
 The direct image functor $j_*\:U\Qcoh\rarrow X\Qcoh$ is exact
(since the morphism~$j$ is affine); so we have a short exact
sequence of direct images $0\rarrow j_*\K\rarrow j_*\G\rarrow
j_*j^*\cH\rarrow0$ in $X\Qcoh$.

 Consider the adjunction morphism of quasi-coherent sheaves
$\cH\rarrow j_*j^*\cH$ on~$X$.
 Let $0\rarrow j_*\K\rarrow\cL\rarrow\cH\rarrow0$ be the pullback
of the short exact sequence $0\rarrow j_*\K\rarrow j_*\G\rarrow
j_*j^*\cH\rarrow0$ with respect to the morphism $\cH\rarrow j_*j^*\cH$
in $X\Qcoh$.
 Put $W=U\cup V\subset X$; so $W$ is an open subscheme in~$X$.
 We claim that the quasi-coherent sheaf $\cL|_W$ is very flat on~$W$.

 Indeed, it suffices to check that the quasi-coherent sheaf
$\cL|_U$ is very flat on $U$ and the quasi-coherent sheaf $\cL|_V$
is very flat on~$V$.
 The exact functor $j^*\:X\Qcoh\rarrow U\Qcoh$ takes the morphism
$\cH\rarrow j_*j^*\cH$ to the identity map $j^*\cH\rarrow j^*\cH$;
so we have $j^*\cL\simeq j^*j_*\G\simeq\G$.
 The quasi-coherent sheaf $\G$ on $U$ is very flat by construction.
 This proves that the quasi-coherent sheaf $\cL|_U=j^*\cL$ is very flat.

 To prove that the quasi-coherent sheaf $\cL|_V$ is very flat,
denote by $h\:V\rarrow X$ the open embedding morphism.
 Then we have a short exact sequence $0\rarrow h^*j_*\K\rarrow
h^*\cL\rarrow h^*\cH\rarrow0$ of quasi-coherent sheaves on~$V$.
 As the quasi-coherent sheaf $\cH|_V=h^*\cH$ is very flat by
assumption, and the class of very flat quasi-coherent sheaves on $V$
is closed under extensions, it suffices to check that
the quasi-coherent sheaf $h^*j_*\K$ on $V$ is very flat.

 Consider the intersection $U\cap V$, and denote by
$j'\:U\cap V\rarrow V$ and $h'\:U\cap V\rarrow U$ its open embedding
morphisms.
 So the morphism~$j'$ is affine as a base change of an affine
morphism of schemes.
 We have an (obvious) base change isomorphism $h^*j_*\K\simeq
j'_*h'{}^*\K$ of quasi-coherent sheaves on~$V$.
 Furthermore, the two compositions $jh'$ and $hj'$ are one and the same
morphism $U\cap V\rarrow X$; so we have an isomorphism
$h'{}^*j^*\cH\simeq j'{}^*h^*\cH$ of quasi-coherent sheaves on
$U\cap V$.

 Applying the functor $h'{}^*\:U\Qcoh\rarrow(U\cap\nobreak V)\Qcoh$
to the short exact sequence $0\rarrow\K\rarrow\G\rarrow j^*\cH\
\rarrow0$ in $U\Qcoh$, we obtain a short exact sequence
$0\rarrow h'{}^*\K\rarrow h'{}^*\G\rarrow h'{}^*j^*\cH\rarrow0$
in $(U\cap\nobreak V)\Qcoh$.
 The quasi-coherent sheaf $h'{}^*\G$ on $U\cap V$ is very flat,
since the quasi-coherent sheaf $\G$ on $U$ is very flat.
 The quasi-coherent sheaf $h'{}^*j^*\cH\simeq j'{}^*h^*\cH$ on
$U\cap V$ is very flat, since the quasi-coherent sheaf
$h^*\cH=\cH|_V$ on $V$ is very flat.

 As the class of very flat quasi-coherent sheaves on $U\cap V$ is closed
under kernels of surjective morphisms, we can conclude from
the short exact sequence $0\rarrow h'{}^*\K\rarrow h'{}^*\G\rarrow
h'{}^*j^*\cH\rarrow0$ that the quasi-coherent sheaf $h'{}^*\K$
on $U\cap V$ is very flat.
 Finally, since the morphism $j'\:U\cap V\rarrow V$ is affine and
very flat, we arrive to the conclusion that the quasi-coherent sheaf
$h^*j_*\K\simeq j'_*h'{}^*\K$ on $V$ is very flat.
 We have proved that the quasi-coherent sheaf $\cL|_W$ is very flat
on $W=U\cup V$.

 The composition $\cL\rarrow\cH\rarrow\M$ of surjective morphisms of
quasi-coherent sheaves on $X$ is surjective.
 Having started with a surjective morphism of quasi-coherent sheaves
$\cH\rarrow\M$ such that the quasi-coherent sheaf $\cH|_V$ is very flat
on an open subscheme $V\subset X$, and given an affine open subscheme
$U\subset X$, we have constructed an surjective morphism of
quasi-coherent sheaves $\cL\rarrow\M$ such that the quasi-coherent
sheaf $\cL|_W$ is very flat on the open subscheme $W=U\cup V$.

 Starting with $V=\varnothing$ and $\cH=\M$ as the induction base,
and proceeding in this way by adjoining affine open subschemes
$U=U_\alpha$ from a given finite affine open covering $X=\bigcup_\alpha
U_\alpha$ of the scheme $X$, we arrive to the desired surjective
morphism onto $\M$ from a very flat quasi-coherent sheaf on~$X$.
\end{proof}

\begin{prop} \label{very-flat-generator-prop}
 For any quasi-compact semi-separated scheme $X$, there exists a very
flat quasi-coherent sheaf $\F$ on $X$ such that $\F$ is a generator
of the Grothendieck category $X\Qcoh$.
\end{prop}

\begin{proof}
 The category $X\Qcoh$ of quasi-coherent sheaves on any scheme $X$
has a generator by Gabber's theorem~\cite[Lemma~2.1.7]{Con},
\cite[Proposition Tag~077P]{SP}.
 More specifically, for any quasi-compact quasi-separated scheme $X$,
the category $X\Qcoh$ is locally finitely presentable by
Grothendieck's theorem~\cite[0.5.2.5 and Corollaire~I.6.9.12]{EGA1}.
 So, in the situation at hand with $X$ a quasi-compact semi-separated
scheme, let $\M$ be a generator of $X\Qcoh$.
 By Lemma~\ref{very-flat-qcoh}, there exists a very flat quasi-coherent
sheaf $\F$ on $X$ together with a surjective morphism of quasi-coherent
sheaves $\F\rarrow\M$.
 Then $\F$ is a desired very flat generator of $X\Qcoh$.
\end{proof}

 Given an object $F$ in an abelian category $\sA$, one says that
$F$ has \emph{projective dimension~$\le d$} if $\Ext_\sA^m(F,A)=0$
for all objects $A\in\sA$ and all integers $m>d$.
 Notice that this definition does \emph{not} presume existence of
enough projective objects in~$\sA$.

\begin{lem} \label{very-flat-sheaf-projective-dimension}
 Let $X$ be a quasi-compact semi-separated scheme with a finite affine
open covering $X=\bigcup_{\alpha=1}^N U_\alpha$.
 Then any very flat quasi-coherent sheaf $\F$ on $X$ has projective
dimension at most~$N$ (as an object of $X\Qcoh$).
\end{lem}

\begin{proof}
 Let $\F$ be an arbitrary quasi-coherent sheaf on~$X$.
 Then, by~\cite[Theorem~6.3(b)]{PS6}, the projective dimension of $\F$
in $X\Qcoh$ does not exceed $N-1$ plus the supremum of the projective
dimensions of the $\cO(U_\alpha)$\+modules $\F(U_\alpha)$.
 When $\F$ is a very flat quasi-coherent sheaf,
the $\cO(U_\alpha)$\+module $\F(U_\alpha)$ is very flat, so its
projective dimension does not exceed~$1$ (according to
Section~\ref{cta-vfl-modules-subsecn}).
\end{proof}

\subsection{Generators of finite projective dimension}
\label{generators-of-fpd-subsecn}
 Given a class of objects $\sF$ in an abelian category $\sA$, we denote
by $\sF^{\perp_{\ge1}}\subset\sA$ the class of all objects $C\in\sA$
such that $\Ext^m_\sA(F,C)=0$ for all $F\in\sF$ and $m\ge1$.
 In particular, for a single object $F\in\sA$, we put
$F^{\perp_{\ge1}}=\{F\}^{\perp_{\ge1}}\subset\sA$.

\begin{lem} \label{orthogonal-to-generator-adjusted-to-products}
 Let\/ $\sA$ be an abelian category with infinite products and enough
injective objects, and let\/ $\Lambda$ be an infinite set.
 Let $F\in\sA$ be a generator of\/~$\sA$.
 Consider the class\/ $\sC=F^{\perp_{\ge1}}\subset\sA$ of all objects
$C\in\sA$ such that $\Ext^m_\sA(F,C)=0$ for all $m\ge1$.
 Let\/ $\sK=\sC^\Lambda\subset\sA^\Lambda=\sE$ be the class of all\/
$\Lambda$\+indexed collections of objects from\/~$\sC$.
 Then\/ $\sK$ is a coresolving class of objects in\/ $\sE$ adjusted
to the direct product functor $G=\prod_{\lambda\in\Lambda}\:\sE
\rarrow\sA$ (in the sense of Section~\ref{adjusted-classes-subsecn}).
\end{lem}

\begin{proof}
 One can easily see that $\sC$ is a coresolving class of objects in
$\sA$, and it follows that $\sK$ is a coresolving class of objects
in~$\sE$.
 (This holds for the right $\Ext^{\ge1}$\+orthogonal class
$\sC=\sF^{\perp_{\ge1}}$ to an arbitrary class of objects
$\sF\subset\sA$.)
 To show that the class $\sK$ is adjusted to the direct product
functor $G$, we argue as in~\cite[Lemma~9.3]{PS6}.

 Let $0\rarrow C_\lambda\rarrow B_\lambda\rarrow A_\lambda\rarrow0$
be a $\Lambda$\+indexed family of short exact sequences in $\sA$
with $C_\lambda\in\sC$ for all $\lambda\in\Lambda$.
 In order to show that $0\rarrow\prod_{\lambda\in\Lambda}C_\lambda
\rarrow\prod_{\lambda\in\Lambda}B_\lambda\rarrow
\prod_{\lambda\in\Lambda}A_\lambda\rarrow0$ is a short exact sequence
in $\sA$, we need to check that $\prod_{\lambda\in\Lambda}B_\lambda
\rarrow\prod_{\lambda\in\Lambda}A_\lambda$ is an epimorphism.
 As $F$ is a generator of $\sA$, it suffices to check that
every morphism $F\rarrow\prod_{\lambda\in\Lambda}A_\lambda$ can be
lifted to a morphism $F\rarrow\prod_{\lambda\in\Lambda}B_\lambda$.
 Now the datum of a morphism $F\rarrow\prod_{\lambda\in\Lambda}
A_\lambda$ is equivalent to the datum of a family of morphisms
$F\rarrow A_\lambda$, \,$\lambda\in\Lambda$.
 Since $\Ext^1_\sA(F,C_\lambda)=0$ by assumption, any morphism
$F\rarrow A_\lambda$ can be lifted to a morphism $F\rarrow B_\lambda$.
 Any family of such liftings defines the desired morphism
$F\rarrow\prod_{\lambda\in\Lambda}B_\lambda$.
\end{proof}

\begin{lem} \label{coresolution-dimension-for-orthogonal-class}
 Let\/ $\sA$ be an abelian category and\/ $\sF\subset\sA$ be a class of
(some) objects of projective dimension\/~$\le d$ in\/~$\sA$.
 Then every object of\/ $\sA$ has coresolution dimension at most~$d$
with respect to the coresolving class\/ $\sC=\sF^{\perp_{\ge1}}\subset
\sA$ of all objects $C\in\sA$ such that\/ $\Ext^m_\sA(F,C)=0$
for all $F\in\sF$ and $m\ge1$.
\end{lem}

\begin{proof}
 This is straightforward.
\end{proof}

\begin{prop} \label{generator-of-finite-projective-dimension-prop}
 Let\/ $\sA$ be an abelian category with infinite products and enough
injective objects, and let\/ $\Lambda$ be an infinite set.
 Let $F\in\sA$ be a generator of\/~$\sA$; assume that the projective
dimension of $F$ in\/ $\sA$ does not exceed~$d$.
 Consider the direct product functor\/ $G=\prod_{\lambda\in\Lambda}\:
\sE=\sA^\Lambda\rarrow\sA$.
 Then one has\/ $\boR^iG=0$ for all $i>d$.
 So the abelian category\/ $\sA$ satisfies Roos' axiom\/
$\AB4^*$\+$d$.
\end{prop}

\begin{proof}
 This is a weak version of a resulf of Roos~\cite[Theorem~1.3]{Roos}.
 As the proof in~\cite{Roos} is somewhat complicated, we offer
a simple alternative proof of the weak version stated in
the proposition.

 Put $\sC=F^{\perp_{\ge1}}\subset\sA$ and
$\sK=\sC^\Lambda\subset\sE$.
 Then the coresolving class $\sK$ is adjusted to the functor $G$ by
Lemma~\ref{orthogonal-to-generator-adjusted-to-products}, and
the coresolution dimension of any object of $\sE$ with respect
to the class $\sK$ does not exceed~$d$ by
Lemma~\ref{coresolution-dimension-for-orthogonal-class}.
 The desired assertion follows by virtue of
Lemma~\ref{adjusted-class-lemma}.
\end{proof}

\subsection{Roos axiom for quasi-compact semi-separated schemes~II}
 The following corollary is a weaker version of
Theorem~\ref{roos-axiom-cech-coresolution-semi-separated}.
 We formulate and prove it here in order to illusrate the workings
of the approach based on existence of a generator of finite
projective dimension.

\begin{cor} \label{roos-axiom-very-flat-generator-semi-separated}
 Let $X$ be a quasi-compact semi-separated scheme with an affine open
covering $X=\bigcup_{\alpha=1}^N U_\alpha$.
 Then the Grothendieck abelian category of quasi-coherent sheaves
$X\Qcoh$ satisfies the axiom\/ $\AB4^*$\+$N$.
\end{cor}

\begin{proof}
 By Proposition~\ref{very-flat-generator-prop}, there exists
a very flat quasi-coherent sheaf $\F$ on $X$ such that $\F$ is
a generator of $X\Qcoh$.
 Lemma~\ref{very-flat-sheaf-projective-dimension} tells us that
the projective dimension of the object $\F\in X\Qcoh$ does not
exceed~$N$.
 So Proposition~\ref{generator-of-finite-projective-dimension-prop}
is applicable for $\sA=X\Qcoh$ and $d=N$.
\end{proof}

\begin{rem} \label{locally-countably-presented}
 The proof of
Corollary~\ref{roos-axiom-very-flat-generator-semi-separated} above
is based on the notion of a very flat quasi-coherent sheaf.
 Alternatively, there is a similar argument proving the same theorem
based on the notion of a \emph{locally countably presented} flat
quasi-coherent sheaf.
 A quasi-coherent sheaf $\M$ on a scheme $X$ is said to be
\emph{locally countably presented}~\cite[Section~3]{PS6} if
the $\cO_X(U)$\+module $\M(U)$ is countably presented for every
affine open subscheme $U\subset X$.
 It suffices to check this condition for affine open subschemes
$U=U_\alpha$ forming a given affine open covering
$X=\bigcup_\alpha U_\alpha$ of a scheme~$X$ \,\cite[Lemma~3.1]{PS6}.
 By~\cite[Theorem~3.5 or~4.5]{PS6}, any flat quasi-coherent sheaf on
a quasi-compact quasi-separated (or more generally, countably
quasi-compact and countably quasi-separated) scheme $X$ is a direct
limit (in fact, an $\aleph_1$\+direct limit) of locally countably
presented flat quasi-coherent sheaves.
 
 For a quasi-compact semi-separated scheme $X$, it is known
since~\cite[Section~2.4]{M-n} (see~\cite[Lemma~A.1]{EP} for a perhaps
clearer argument) that any quasi-coherent sheaf $\M$ on $X$ is
a quotient sheaf of a flat quasi-coherent sheaf~$\G$.
 Then $\G$ is a direct limit of locally countably presented flat
quasi-coherent sheaves.
 So $\M$ is a quotient sheaf of an infinite direct sum of locally
countably presented flat quasi-coherent sheaves.
 Take $\M$ to be a generator of $X\Qcoh$.
 Any countably presented flat module over a ring has projective
dimension at most~$1$ \,\cite[Corollary~2.23]{GT}.
 If $X$ covered by $N$ open affines, then it follows by virtue
of~\cite[Theorem~6.3(b)]{PS6} that any locally countably presented
flat quasi-coherent sheaf on $X$ has projective dimension at most~$N$
\,\cite[Corollary~6.6]{PS6}.
 Hence the same applies to any infinite direct sum of locally
countably presented flat quasi-coherent sheaves on~$X$.
 Once again, we have constructed a generator of projective dimension
at most~$N$ for $X\Qcoh$, and it remains to refer to
Proposition~\ref{generator-of-finite-projective-dimension-prop} in
order to finish the proof of the theorem.
\end{rem}

\begin{rem}
 The argument from Remark~\ref{locally-countably-presented} allows
also to prove the Roos axiom for the categories $X\Qcoh$ of
quasi-coherent sheaves on certain kind of (quasi-compact
semi-separated) stacks~$X$.
 Specifically, the categories $X\Qcoh$ for the stacks we are
interested in here are interpreted as the abelian categories of
\emph{comodules} over certain corings~$C$ over commutative (or
more generally, associative) rings~$A$; so $X\Qcoh=C\Comodl$.
 The conditions on a coring $C$ guaranteeing existence of
a generator of finite projective dimension in $C\Comodl$ are
spelled out in the paper~\cite{Pflcc} as~\cite[conditions~($*$)
and~(${*}{*}$) in Section~5]{Pflcc}, and the related argument
can be found in~\cite[last paragraph of the proof of
Corollary~5.5]{Pflcc}.
 For a discussion of the connection between stacks and corings,
see~\cite[Section~2]{KR} or~\cite[Example~2.5]{Pflcc}.
\end{rem}

\begin{quest}
 Is the dimension bound provided by
Corollary~\ref{roos-axiom-very-flat-generator-semi-separated} sharp?
 This question was asked in the early first version of this paper,
when the author was unaware of the preprint~\cite{HX}, and the second
version of the preprint~\cite{HPSV} (with its~\cite[Theorem~8.27]{HPSV})
did not appear yet.
 Now we know that the answer is negative:
Theorem~\ref{roos-axiom-cech-coresolution-semi-separated} improves
the numerical bound of 
Corollary~\ref{roos-axiom-very-flat-generator-semi-separated}.

 Still, it would be interesting to have a natural, well-behaved class
of quasi-coherent sheaves adjusted to infinite products in $X\Qcoh$
and providing finite coresolutions of length $N-1$, where
a quasi-compact semi-separated scheme $X$ is covered by $N$ affine
open subschemes.
 This would provide an alternative proof of
Theorem~\ref{roos-axiom-cech-coresolution-semi-separated}
in view of Lemma~\ref{adjusted-class-lemma}.

 Our remaining question is: does the class of \emph{dilute
quasi-coherent sheaves} \cite[Section~4.1]{M-n},
\cite[Section~4.2]{Pcosh} satisfy these conditions?
 Specifically, are all families of dilute quasi-coherent sheaves
adjusted to infinite products in $X\Qcoh$\,?
 Notice that the coresolution dimension of any quasi-coherent sheaf on
$X$ with respect to the coresolving subcategory of dilute quasi-coherent
sheaves does not exceed $N-1$ \,\cite[Lemma~4.7.1(a)]{Pcosh}.

 A natural, well behaved class of quasi-coherent sheaves adjusted to
infinite products in $X\Qcoh$ and providing finite coresolutions of
length $N$ is suggested in the following
Section~\ref{contraadjusted-subsecn}.
\end{quest}

\subsection{Contraadjusted quasi-coherent sheaves}
\label{contraadjusted-subsecn}
 Given a quasi-compact semi-separated scheme $X$, the discussion in
Section~\ref{generators-of-fpd-subsecn} provides a coresolving class
$\sC\subset\sA=X\Qcoh$ such that all the objects of $\sA$ have finite
(uniformly bounded) $\sC$\+coresolution dimensions and the class
$\sK=\sC^\Lambda\subset\sA^\Lambda=\sE$ is adjusted to the functor
of direct product $G=\prod_{\lambda\in\Lambda}\:\sE\rarrow\sA$.
 But this construction is not natural, in that it depends on the choice
of a very flat generator $\F\in X\Qcoh$.

 In this section we discuss a straightforward modification or particular
case of the construction of Section~\ref{generators-of-fpd-subsecn}
providing a natural coresolving class in $X\Qcoh$ adjusted to infinite
products and such that the coresolution dimensions are finite.
 This class is called the class of \emph{contraadjusted} quasi-coherent
sheaves.

 Let $X$ be a scheme.
 We will denote by $\Ext_X^*({-},{-})$ the functor $\Ext^*$ in
the abelian category of quasi-coherent sheaves $X\Qcoh$.

 A quasi-coherent sheaf $\C$ on $X$ is said to be
\emph{contraadjusted}~\cite[Section~2.5]{Pcosh} if $\Ext^1_X(\F,\C)=0$
for every very flat quasi-coherent sheaf $\F$ on~$X$.
 For example, any cotorsion quasi-coherent sheaf on~$X$ (in the sense
of~\cite{EE}) is contraadjusted, since any very flat quasi-coherent
sheaf is flat.
 A quasi-coherent sheaf on an affine scheme $U$ is contraadjusted if
and only if it corresponds to a contraadjusted $\cO(U)$\+module.

 By the definition, the class of contraadjusted quasi-coherent sheaves
is closed under extensions and direct summands in $X\Qcoh$.
 For any quasi-compact quasi-separated morphism of schemes
$f\:Y\rarrow X$, the direct image functor $f_*\:Y\Qcoh\rarrow X\Qcoh$
takes contraadjusted quasi-coherent sheaves on $Y$ to contraadjusted
quasi-coherent sheaves on~$X$ \,\cite[Section~2.5]{Pcosh},
\cite[Lemma~1.7(c)]{Pal}.
 
 The following two propositions say that the pair of classes (very flat
quasi-coherent sheaves, contraadjusted quasi-coherent sheaves) is
a hereditary complete cotorsion pair in $X\Qcoh$ for a quasi-compact
semi-separated scheme~$X$.

\begin{prop} \label{vfl-cotorsion-pair-on-scheme}
 Let $X$ be a quasi-compact semi-separated scheme.
 Then \par
\textup{(a)} any quasi-coherent sheaf on $X$ is a quotient sheaf of
a very flat quasi-coherent sheaf by a subsheaf that is
a contraadjusted quasi-coherent sheaf; \par
\textup{(b)} any quasi-coherent sheaf on $X$ can be embedded as
subsheaf into a contraadjusted quasi-coherent sheaf in such a way
that the quotient sheaf is a very flat quasi-coherent sheaf.
\end{prop}

\begin{prop} \label{vfl-cotorsion-pair-on-scheme-hereditary}
 Let $X$ be a quasi-compact semi-separated scheme.
 Then \par
\textup{(a)} the class of contraadjusted quasi-coherent sheaves on
$X$ is closed under cokernels of injective morphisms; \par
\textup{(b)} for any very flat quasi-coherent sheaf $\F$ and any
contraadjusted quasi-coherent sheaf $\C$ on $X$, one has\/
$\Ext_X^m(\F,\C)=0$ for all $m\ge1$.
\end{prop}

 The third proposition provides an additional description of the class
of contraadjusted quasi-coherent sheaves on~$X$.

\begin{prop} \label{contraadjusted-sheaves-described}
 Let $X$ be a quasi-compact semi-separated scheme with a finite affine
open covering $X=\bigcup_\alpha U_\alpha$.
 Denote by $j_\alpha\:U_\alpha\rarrow X$ the open embedding morphisms.
 Then a quasi-coherent sheaf $\C$ on $X$ is contraadjusted if and only
if $\C$ is a direct summand of a finitely iterated extension in
$X\Qcoh$ of the direct images $j_\alpha{}_*\C_\alpha$ of contraadjusted
quasi-coherent sheaves $\C_\alpha$ on~$U_\alpha$.
\end{prop}

\begin{proof}[Proof of
Propositions~\ref{vfl-cotorsion-pair-on-scheme},
\ref{vfl-cotorsion-pair-on-scheme-hereditary},
and~\ref{contraadjusted-sheaves-described}]
 All these assertions can be found in~\cite[Section~4.1]{Pcosh}.
 Proposition~\ref{vfl-cotorsion-pair-on-scheme}
is~\cite[Corollary~4.1.4(a\+-b)]{Pcosh}.
 Proposition~\ref{vfl-cotorsion-pair-on-scheme-hereditary}
is~\cite[Corollary~4.1.2(b\+-c)]{Pcosh}.
 Proposition~\ref{contraadjusted-sheaves-described}
is~\cite[Corollary~4.1.4(c)]{Pcosh}.

 Let us provide a little discussion with more details.
 To prove Proposition~\ref{vfl-cotorsion-pair-on-scheme}(a), one
just needs to follow the proof of Lemma~\ref{very-flat-qcoh} above
carefully and look into what one can say about the kernel of
the surjective morphism of quasi-coherent sheaves that the construction
from the proof of the lemma provides.
 In the course of the proof of the lemma, when choosing a presentation
of the quasi-coherent sheaf $j^*\cH$ on $U$ as a quotient of a very
flat quasi-coherent sheaf $\G$, one needs to do it in such a way that
the kernel $\K$ is a contraadjusted quasi-coherent sheaf on~$U$ (which
is possible by~\cite[Theorem~1.1.1(b)]{Pcosh}).
 One will see that not only the kernel of the resulting surjective
morphism from a very flat quasi-coherent sheaf on $X$ is contraadjusted,
but also it has the form of a finitely iterated extension described in
Proposition~\ref{contraadjusted-sheaves-described}.
 The rest of the proof of all the three propositions is easier and
more standard.
 In particular,
Proposition~\ref{vfl-cotorsion-pair-on-scheme-hereditary} follows
directly from the assertion of Lemma~\ref{very-flat-qcoh} together
with the fact that the class of very flat quasi-coherent sheaves
is closed under kernels of epimorphisms in $X\Qcoh$.

 The results of Propositions~\ref{vfl-cotorsion-pair-on-scheme}\+-%
\ref{contraadjusted-sheaves-described} form a special case of
a general construction of gluing of hereditary complete cotorsion
pairs from an affine open covering of a quasi-compact semi-separated
scheme $X$, as described in~\cite[Section~B.4]{Pcosh}.
 For a version of this construction for principal affine open coverings
of affine schemes, see~\cite[Proposition~4.3]{Pal}.
\end{proof}

 Proposition~\ref{vfl-cotorsion-pair-on-scheme-hereditary}(a) implies
that the class of contraadjusted quasi-coherent sheaves is coresolving
in $X\Qcoh$ when $X$ is a quasi-compact semi-separated scheme.
 The coresolution dimension of a quasi-coherent sheaf $\M$ on $X$
with respect to the coresolving class of contraadjusted quasi-coherent
sheaves $\sC$ in $X\Qcoh$ is called the \emph{contraadjusted dimension}
of~$\M$.

\begin{lem} \label{contraadjusted-dimension-lemma}
 Let $X$ be a quasi-compact semi-separated scheme with an affine open
covering $X=\bigcup_{\alpha=1}^N U_\alpha$.
 Then the contraadjusted dimension of any quasi-coherent sheaf on $X$
does not exceed~$N$.
\end{lem}

\begin{proof}
 This is~\cite[Lemma~4.7.1(b)]{Pcosh}.
 The assertion is easily provable by comparing
Lemma~\ref{very-flat-sheaf-projective-dimension} with
Lemma~\ref{coresolution-dimension-for-orthogonal-class}
(for $d=N$ and the class $\sF$ of all very flat quasi-coherent sheaves
in $\sA=X\Qcoh$).
 One also has to keep in mind
Proposition~\ref{vfl-cotorsion-pair-on-scheme-hereditary}(b).
\end{proof}

\begin{lem} \label{contraadjusted-product-adjusted-lemma}
 Let $X$ be a quasi-compact semi-separated scheme and\/
$\Lambda$ be a set.
 Put\/ $\sA=X\Qcoh$, and consider the direct product functor
$G=\prod_{\lambda\in\Lambda}\:\sE=\sA^\Lambda\rarrow\sA$.
 Let\/ $\sC\subset\sA$ be the class of all contraadjusted
quasi-coherent sheaves on~$X$.
 Then the class\/ $\sK=\sC^\Lambda\subset\sE$ of all\/
$\Lambda$\+indexed collections of contraadjusted quasi-coherent
sheaves on $X$ is adjusted to the functor~$G$.

 In other words, let\/ $0\rarrow\C_\lambda\rarrow\B_\lambda
\rarrow\A_\lambda\rarrow0$ be a $\Lambda$\+indexed collection
of short exact sequences of quasi-coherent sheaves on $X$ with
contraadjusted quasi-coherent sheaves $\C_\lambda$,
$\lambda\in\Lambda$.
 Then the short sequence of direct products
$$
 0\lrarrow\prod\nolimits_{\lambda\in\Lambda}\C_\lambda\lrarrow
 \prod\nolimits_{\lambda\in\Lambda}\B_\lambda\lrarrow
 \prod\nolimits_{\lambda\in\Lambda}\A_\lambda\lrarrow0
$$
is exact in $X\Qcoh$.
\end{lem}

\begin{proof}
 The argument from the proof of
Lemma~\ref{orthogonal-to-generator-adjusted-to-products} is applicable
in view of Proposition~\ref{very-flat-generator-prop}.
\end{proof}

 Comparing Lemma~\ref{adjusted-class-lemma}
with Lemmas~\ref{contraadjusted-dimension-lemma}
and~\ref{contraadjusted-product-adjusted-lemma}, we obtain another
(slightly differently structured) proof of
Corollary~\ref{roos-axiom-very-flat-generator-semi-separated}.

\Section{Co-Contra Correspondence Argument}  \label{co-contra-secn}

 In this section we present an abstract category-theoretic version of
the argument of~\cite[Section~6.10]{Pcosh} deducing the Roos axiom
for an exact category $\sA$ from the existence of a ``na\"\i ve
co-contra correspondence'' of $\sA$ with an exact category $\sB$ with
exact direct products.

\subsection{Cohomologically bounded complexes in exact categories}
 We suggest the survey paper~\cite{Bueh} as the background reference
source on exact categories in the sense of Quillen.
 To simplify the exposition, in this paper we assume all our exact
categories $\sE$ to be idempotent-complete.
 This means that any idempotent endomorphism of an object $E\in\sE$
is the projection onto a direct summand $F$ in a direct sum
decomposition $E=F\oplus G$ in~$\sE$ \,\cite[Section~6]{Bueh}.

 Let $\sE$ be an idempotent-complete exact category.
 A complex $E^\bu$ in $\sE$ is said to be
\emph{acyclic}~\cite[Definition~10.1]{Bueh} if it can be obtained
by splicing (admissible) short exact sequences $0\rarrow Z^i\rarrow
E^i\rarrow Z^{i+1}\rarrow0$ (so that the morphism $E^i\rarrow E^{i+1}$
is equal to the composition $E^i\rarrow Z^{i+1}\rarrow E^{i+1}$).
 Notice that, even though one can speak of acyclic complexes in
an exact category, the notion of cohomology objects of a complex in
an exact category does not make sense in general.
 Nevertheless, the following definition is meaningful.

 Given a complex $E^\bu$ in $\sE$ and an integer $d\in\boZ$, let us say
that \emph{$E^\bu$ is acyclic in cohomological degrees~$>d$} and write
$H^{>d}(E^\bu)=0$ if the following condition holds.
 There should exist a sequence of (admissible) short exact sequences
$0\rarrow Z^i\rarrow E^i\rarrow Z^{i+1}\rarrow0$ in $\sE$, defined for
all $i\ge d$, such that the morphism $E^i\rarrow E^{i+1}$ is equal to
the composition $E^i\rarrow Z^{i+1}\rarrow E^{i+1}$ for every $i\ge d$.
 
 Dually, we will say that \emph{$E^\bu$ is acyclic in cohomological
degrees~$<-d$} and write $H^{<-d}(E^\bu)=0$ if there exists a sequence
of short exact sequences $0\rarrow Z^i\rarrow E^i\rarrow Z^{i+1}
\rarrow0$, defined for all $i\le-d$, such that the morphism $E^{i-1}
\rarrow E^i$ is equal to the composition $E^{i-1}\rarrow Z^i\rarrow
E^i$ for every $i\le-d$.

 The definition of the unbounded derived category $\sD(\sE)$ of
an exact category $\sE$ can be found in~\cite[Section~10.4]{Bueh}.
 We will say colloquially that two complexes in $\sE$ are
``quasi-isomorphic'' if they are isomorphic as objects of $\sD(\sE)$.
 More specifically, a morphism of complexes in $\sE$ is said to be
a \emph{quasi-isomorphism} if its cone is acyclic.

\begin{lem} \label{bounded-cohomology-invariant}
 Any complex in\/ $\sE$ that is quasi-isomorphic to a complex acyclic
in cohomological degrees~$>d$ is also acyclic in cohomological
degrees~$>d$.
 Dually, any complex quasi-isomorphic to a complex acyclic in 
cohomological degrees~$<-d$ is also acyclic in cohomological
degrees~$<-d$.
\end{lem}

\begin{proof}
 The argument from~\cite[proof of Lemma~A.5.1]{Pcosh} proves the second
assertion (while the first one is dual).
 Essentially, the point is that a complex $E^\bu$ in $\sE$ is acyclic
in cohomological degrees~$<\nobreak-d$ if and only if it is
quasi-isomorphic to a complex $F^\bu$ such that $F^i=0$ for all $i<-d$.
 Here the ``only if'' claim is obvious: take $F^\bu$ to be the complex
$0\rarrow Z^{-d+1}\rarrow E^{-d+1}\rarrow E^{-d+2}\rarrow\dotsb$
(where the morphism $Z^{-d+1}\rarrow E^{-d+1}$ is obtained by
factorizing uniquely the morphism $E^{-d}\rarrow E^{-d+1}$ as
$E^{-d}\rarrow Z^{-d+1}\rarrow E^{-d+1}$, using the fact that
the composition $Z^{-d}\rarrow E^{-d}\rarrow E^{-d+1}$ vanishes).
 Then there is an obvious (termwise admissibly epimorphic)
quasi-isomorphism $E^\bu\rarrow F^\bu$.

 To prove the ``if'', it suffices to check two statements:
\begin{enumerate}
\item if $F^i=0$ for all $i<-d$ and $G^\bu\rarrow F^\bu$ is
a quasi-isomorphism, then $G^\bu$ is acyclic in cohomological
degrees~$<-d$;
\item if $F^i=0$ for all $i<-d$ and $F^\bu\rarrow G^\bu$ is
a quasi-isomorphism, then $G^\bu$ is acyclic in cohomological
degrees~$<-d$.
\end{enumerate}
 Then one argues as follows.
 Given an arbitrary complex $E^\bu$ isomorphic to $F^\bu$ in
$\sD(\sE)$ with $F^i=0$ for all $i<-d$, there exists a ``roof'' of
two quasi-isomorphisms $E^\bu\rarrow G^\bu\larrow F^\bu$.
 By~(2), it follows that there are a complex ${}'\!F^\bu$ in
$\sE$ with ${}'\!F^i=0$ for all $i<-d$ and a quasi-isomorphism
$G^\bu\rarrow{}'\!F^\bu$.
 Hence the composition $E^\bu\rarrow G^\bu\rarrow{}'\!F^\bu$
is also a quasi-isomorphism, and it remains to apply~(1).

 We refer to~\cite[proof of Lemma~A.5.1]{Pcosh} for the details
concerning the proofs of~(1) and~(2).
 Notice that in the setting of~\cite[Section~A.5]{Pcosh} the exact
categories are not assumed to be idempotent-complete, but only
\emph{weakly idempotent-complete} (as in~\cite[Section~7]{Bueh}).
 This necessitates requiring the complexes in~\cite[Lemma~A.5.1]{Pcosh}
to be bounded above, which is not needed in our context.
\end{proof}

\begin{lem} \label{direct-summand-cohomologically-bounded}
 Any complex in\/ $\sE$ that is isomorphic, as an object of\/
the derived category\/ $\sD(\sE)$, to a direct summand of a complex
acyclic in cohomological degrees~$>d$ is also acyclic in cohomological
degrees~$>d$.
 Dually, any complex in\/ $\sE$ that is isomorphic, as an object of\/
the derived category\/ $\sD(\sE)$, to a direct summand of a complex
acyclic in cohomological degrees~$<-d$ is also acyclic in cohomological
degrees~$<-d$.
\end{lem}

\begin{proof}
 Let us prove the second assertion.
 Let $E^\bu$ be a complex in $\sE$ such that $H^{<-d}(E^\bu)=0$.
 Suppose that $E^\bu$ is isomorphic to a direct sum $C^\bu\oplus D^\bu$
in $\sD(\sE)$ (where $C^\bu$ and $D^\bu$ are two complexes in~$\sE$).
 By Lemma~\ref{bounded-cohomology-invariant}, it follows that
$H^{<-d}(C^\bu\oplus D^\bu)=0$.
 Using the fact that any direct summand of an admissible morphism in
an idempotent-complete exact category is an admissible
morphism~\cite[Definition~8.1 and Exercise~8.18]{Bueh}, one easily
arrives to the conclusion that $H^{<-d}(C^\bu)=0$.
\end{proof}

\subsection{Exact categories with exact products}
 Given an exact category $\sE$, we denote by $\sK(\sE)$ the cochain
homotopy category of complexes in~$\sE$.

 Given a set $\Lambda$ and a family of complexes
$(E^\bu_\lambda)_{\lambda\in\Lambda}$ in $\sE$, one has to distinguish
between the direct product of $E^\bu_\lambda$ taken in the homotopy
category $\sK(\sE)$ and the direct product of the same complexes
computed in the derived category $\sD(\sE)$.
 Generally speaking, they do \emph{not} agree, if the direct product
functors are not exact in~$\sE$.

 If direct products exist in $\sE$, then the direct products in
the homotopy category $\sK(\sE)$ can be computed termwise (while
the direct products in $\sD(\sE)$ are certain derived functors).
 We denote the respective products by
$\prod_{\lambda\in\Lambda}^{\sK(\sE)}E^\bu_\lambda$
and $\prod_{\lambda\in\Lambda}^{\sD(\sE)}E^\bu_\lambda$.

 Let $\sB$ be an exact category where all infinite products exist.
 We will say that $\sB$ has \emph{exact direct products} if
the termwise direct product of any family of (admissible) short exact
sequences is a short exact sequence.
 Equivalently, this means that the termwise direct product of any
family of acyclic complexes is an acyclic complex.

\begin{lem} \label{exact-products}
 Let\/ $\sB$ be an exact category with exact direct products.
 Then the direct products in the homotopy category\/ $\sK(\sB)$ and
in the derived category\/ $\sD(\sB)$ agree.
 In other words, for any family of complexes
$(B^\bu_\lambda)_{\lambda\in\Lambda}$ in $\sB$, the direct product
of $B^\bu_\lambda$ in\/ $\sK(\sB)$ is also their direct product in\/
$\sD(\sB)$, that is\/
$\prod_{\lambda\in\Lambda}^{\sD(\sB)}B^\bu_\lambda=
\prod_{\lambda\in\Lambda}^{\sK(\sB)}B^\bu_\lambda$.
\end{lem}

\begin{proof}
 This is an instance of the general fact that the Verdier quotient
functor of a triangulated category with infinite products by
a thick subcategory closed under products preserves infinite
products~\cite[Lemma~1.5]{BN}.
\end{proof}

\subsection{Bounded below complexes of injectives}
 We refer to~\cite[Section~11]{Bueh} for a discussion of injective
objects in exact categories.

\begin{lem} \label{bounded-below-are-homotopy-injective}
 Let\/ $\sA$ be an exact category and $0\rarrow J^0\rarrow J^1\rarrow
J^2\rarrow\dotsb$ be a bounded below complex of injective objects
in\/~$\sA$.
 Then, for any complex $C^\bu$ in\/ $\sA$, the Verdier quotient functor
$\sK(\sA)\rarrow\sD(\sA)$ induces an isomorphism of the\/ $\Hom$ groups
$$
 \Hom_{\sK(\sA)}(C^\bu,J^\bu)\rarrow\Hom_{\sD(\sA)}(C^\bu,J^\bu).
$$
\end{lem}

\begin{proof}
 The assertion holds because, for any acyclic complex $A^\bu$ in $\sA$,
any morphism of complexes $A^\bu\rarrow J^\bu$ is homotopic to zero.
 The latter statement is provable by a well known standard argument
proceeding by induction on the cohomological
degree~\cite[Corollary~12.6]{Bueh}.
\end{proof}

\begin{lem} \label{bounded-below-complexes-of-injectives-products}
 Let\/ $\sA$ be an exact category with infinite products, $\Lambda$ be
a set, and $0\rarrow J^0_\lambda\rarrow J^1_\lambda\rarrow J^2_\lambda
\rarrow\dotsb$ be a family of uniformly bounded below complexes of
injective objects in\/ $\sA$, indexed by $\lambda\in\Lambda$.
 Then the direct product of the complexes $J_\lambda^\bu$ in\/
$\sK(\sA)$ is also their direct product in\/ $\sD(\sA)$, that is\/
$\prod_{\lambda\in\Lambda}^{\sD(\sA)}J^\bu_\lambda=
\prod_{\lambda\in\Lambda}^{\sK(\sA)}J^\bu_\lambda$.
\end{lem}

\begin{proof}
 It is helpful to notice that infinite products of injective objects
are injective in any exact category.
 The assertion follows from
Lemma~\ref{bounded-below-are-homotopy-injective}.
\end{proof}

\subsection{Direct summand argument}
 The following lemma is a straightforward but useful device for
existence proofs.

\begin{lem} \label{direct-summand-argument}
\textup{(a)} Let\/ $\sE$ be an exact category with infinite products.
 Then, for any set\/ $\Delta$ and any family of objects
$(E_\delta\in\sE)_{\delta\in\Delta}$ there exists an object $E\in\sE$
such that every object $E_\delta$ is a direct summand of~$E$. \par
\textup{(b)} Let\/ $\sE$ be an exact category with infinite products.
 Then, for any sets\/ $\Lambda$ and\/ $\Delta$ and any family of
families of objects
$(E_{\lambda,\delta}\in\sE)_{\lambda\in\Lambda,\,\delta\in\Delta}$
there exists a family of objects
$(E_\lambda\in\sE)_{\lambda\in\Lambda}$ such every object
$E_{\lambda,\delta}$ is a direct summand of~$E_\lambda$.
\end{lem}

\begin{proof}
 Part~(a): put $E=\prod_{\delta\in\Delta}E_\delta$.
 Part~(b): put $E_\lambda=\prod_{\delta\in\Delta}E_{\lambda,\delta}$
for every $\lambda\in\Lambda$.
\end{proof}

\subsection{Na\"\i ve co-contra correspondence implies Roos axiom}
 Let\/ $\sA$ be an exact category with enough injective
objects (see the dual definition in~\cite[Definition~11.9]{Bueh}).
 Notice that in this case every object $A\in\sA$ has an injective
coresolution $0\rarrow A\rarrow J^0\rarrow J^1\rarrow J^2
\rarrow\dotsb$ \,\cite[Proposition~12.2]{Bueh}.
 Viewed as an object of the homotopy category $\sK(\sA)$,
an injective coresolution $J^\bu$ of an object $A\in\sA$ is defined
uniquely up to a unique isomorphism~\cite[Corollary~12.5]{Bueh}.

 Assume further that $\sA$ has infinite direct products.
 Given a set $\Lambda$ and a family of objects
$(A_\lambda\in\sA)_{\lambda\in\Lambda}$, choose an injective
coresolution $0\rarrow A_\lambda\rarrow J^0_\lambda\rarrow J^1_\lambda
\rarrow J^2_\lambda\rarrow\dotsb$ for every object~$A_\lambda$.
 Let us say that the exact category $\sA$ \emph{satisfies the Roos
axiom $\AB4^*$\+$n$} for some integer $n\ge0$ if, for every set
$\Lambda$ and every family of objects $A_\lambda\in\sA$,
the complex $\prod^{\sK(\sA)}_{\lambda\in\Lambda}J_\lambda^\bu$ is
acyclic in cohomological degrees~$>n$.

 For any exact category $\sE$, in addition to the unbounded derived
category $\sD(\sE)$, one can also consider the bounded derived
category $\sD^\bb(\sE)$.
 The natural functor $\sD^\bb(\sE)\rarrow\sD(\sE)$ is fully
faithful~\cite[Lemma~11.7]{Kel}.

\begin{thm} \label{roos-axiom-co-contra}
 Let\/ $\sA$ be an idempotent-complete exact category with direct
products and enough injective objects, and let\/ $\sB$ be
an idempotent-complete exact category with \emph{exact} direct
products.
 Assume that there exists a triangulated equivalence\/
$\sD(\sA)\simeq\sD(\sB)$ that restricts to an equivalence between
the full subcategories\/ $\sD^\bb(\sA)\subset\sD(\sA)$ and\/
$\sD^\bb(\sB)\subset\sD(\sB)$, that is\/ $\sD^\bb(\sA)\simeq
\sD^\bb(\sB)$.
 Then there \emph{exists} a finite integer $n\ge0$ for which
the exact category\/ $\sA$ satisfies $\AB4^*$\+$n$.
\end{thm}

\begin{proof}
 Suppose for the sake of contradiction that such a finite integer~$n$
does \emph{not} exist.
 Then, for every integer $\delta\ge0$, there is a set $\Lambda_\delta$
and a family of objects
$(A_{\lambda,\delta}\in\sA)_{\lambda\in\Lambda_\delta}$ such that,
choosing an injective coresolution
$0\rarrow A_{\lambda,\delta}\rarrow J^\bu_{\lambda,\delta}$ for
each object $A_{\lambda,\delta}$, the complex
$\prod_{\lambda\in\Lambda_\delta}^{\sK(\sA)}J^\bu_{\lambda,\delta}$
is \emph{not} acyclic in cohomological degrees~$>\delta$.

 Denote by $\Lambda$ the disjoint union
$\Lambda=\coprod_{\delta=0}^\infty\Lambda_\delta$, and put
$A_{\lambda,\delta}=0$ for all
$\lambda\in\Lambda\setminus\Lambda_\delta$.
 Applying Lemma~\ref{direct-summand-argument}(b), we get
a family of objects $(A_\lambda\in\sA)_{\lambda\in\Lambda}$
such that $A_{\lambda,\delta}$ is a direct summand of $A_\lambda$
for every $\lambda\in\Lambda$ and $\delta\ge0$.
 In other words, for every $\delta\ge0$, the family of objects
$(A_{\lambda,\delta})_{\lambda\in\Lambda}$ is a direct summand
of the family of objects $(A_\lambda)_{\lambda\in\Lambda}$.

 Choose an injective coresolution $0\rarrow A_\lambda\rarrow
J_\lambda^\bu$ for every object~$A_\lambda$.
 Then, for every $\delta\ge0$, the complex
$\prod_{\lambda\in\Lambda_\delta}^{\sK(\sA)}J^\bu_{\lambda,\delta}$,
viewed as an object of the homotopy category $\sK(\sA)$, is
a direct summand of the complex
$\prod_{\lambda\in\Lambda}^{\sK(\sA)}J^\bu_\lambda$.
 Applying Lemma~\ref{direct-summand-cohomologically-bounded},
we arrive to the conclusion that, whatever integer $d\ge0$ one
chooses, the single complex
$\prod_{\lambda\in\Lambda}^{\sK(\sA)}J^\bu_\lambda$ is
\emph{not} acyclic in cohomological degrees~$>d$.

 Denote our triangulated equivalence by
$$
 \Psi\:\sD(\sA)\simeq\sD(\sB):\!\Phi.
$$
 Applying Lemma~\ref{direct-summand-argument}(a), we can find an object
$A\in\sA$ such that, for every $\lambda\in\Lambda$, the object
$A_\lambda$ is a direct summand of $A$ in~$\sA$.
 Then we have $A\in\sA\subset\sD^\bb(\sA)\subset\sD(\sA)$, and
by assumption it follows that $\Psi(A)\in\sD^\bb(\sB)\subset\sD(\sB)$.
 So there exist integers $d_1\ge0$ and $d_2\ge0$ such that
the object $\Psi(A)\in\sD^\bb(\sB)$ can be represented by
a complex $B^\bu$ concentrated in the cohomological degrees
$-d_1\le i\le d_2$, i.~e., $B^i=0$ when $i<-d_1$ or $i>d_2$.
 For every $\lambda\in\Lambda$, the object $\Psi(A_\lambda)\in
\sD(\sB)$ is a direct summand of the object $\Psi(A)\in\sD(\sB)$.
 By Lemma~\ref{direct-summand-cohomologically-bounded}, it follows
that the object $\Psi(A_\lambda)\in\sD(\sB)$ can be represented
by a complex $B^\bu_\lambda$ that is also concentrated in
the cohomological degrees $-d_1\le i\le d_2$, i.~e.,
$B^i_\lambda=0$ for every $\lambda\in\Lambda$ when $i<-d_1$ or $i>d_2$.

 According to
Lemma~\ref{bounded-below-complexes-of-injectives-products}, we have
$\prod^{\sD(\sA)}_{\lambda\in\Lambda}A_\lambda=
\prod^{\sD(\sA)}_{\lambda\in\Lambda}J_\lambda^\bu=
\prod^{\sK(\sA)}_{\lambda\in\Lambda}J_\lambda^\bu$.
 According to
Lemma~\ref{exact-products}, we have
$\prod^{\sD(\sB)}_{\lambda\in\Lambda}B_\lambda^\bu=
\prod^{\sK(\sB)}_{\lambda\in\Lambda}B_\lambda^\bu$.
 As any equivalence of categories, the equivalence
$\sD(\sA)\simeq\sD(\sB)$ preserves infinite products, so
$\Psi(A_\lambda)=B_\lambda^\bu$ implies
$$
 \Psi\left(\prod\nolimits^{\sD(\sA)}_{\lambda\in\Lambda}A_\lambda\right)
 =\prod\nolimits^{\sD(\sB)}_{\lambda\in\Lambda}B_\lambda^\bu.
$$
 We have shown that
$$
 \Psi\left(\prod\nolimits^{\sK(\sA)}_{\lambda\in\Lambda}
 J_\lambda^\bu\right)
 =\prod\nolimits^{\sK(\sB)}_{\lambda\in\Lambda}B_\lambda^\bu.
$$
 Now we have $\prod^{\sK(\sB)}_{\lambda\in\Lambda}B_\lambda^\bu
\in\sD^\bb(\sB)$, hence by assumption
$\prod\nolimits^{\sK(\sA)}_{\lambda\in\Lambda}J_\lambda^\bu\in
\sD^\bb(\sA)\subset\sD(\sA)$.

 In other words, this means that there exists $d\ge0$ such that
the complex $\prod_{\lambda\in\Lambda}^{\sK(\sA)}J^\bu_\lambda$
is acyclic in cohomological degrees~$>d$.
 We have come to a contradiction proving existence of an integer
$n\ge0$ for which the exact category $\sA$ satisfies $\AB4^*$\+$n$.
\end{proof}

\subsection{Roos axiom for finite-dimensional Noetherian schemes~II}
 Now we can briefly discuss the application of
Theorem~\ref{roos-axiom-co-contra} to the categories $X\Qcoh$ of
quasi-coherent sheaves on Noetherian schemes $X$ of finite Krull
dimension.
 The following corollary is a weaker version of
Theorem~\ref{roos-axiom-cech-coresolution-finite-krull-dim}.

\begin{cor} \label{roos-axiom-co-contra-finite-krull-dim}
 For any Noetherian scheme $X$ of finite Krull dimension, there
\emph{exists} an integer $n\ge0$ such that the Grothendieck category
$X\Qcoh$ satisfies $\AB4^*$\+$n$.
\end{cor}

\begin{proof}
 For any scheme $X$, the construction of~\cite[Section~2.2]{Pcosh}
produces an exact category of \emph{contraherent cosheaves}
$X\Ctrh$, which is an idempotent-complete exact category with
exact direct products.
 In the situation at hand with a Noetherian scheme $X$ of finite Krull
dimension, put $\sA=X\Qcoh$ and $\sB=X\Ctrh$.
 Then the result of~\cite[Theorem~6.8.1]{Pcosh} provides
a triangulated equivalence $\sD(\sA)\simeq\sD(\sB)$ that restricts
to a triangulated equivalence $\sD^\bb(\sA)\simeq\sD^\bb(\sB)$.
 It remains to refer to Theorem~\ref{roos-axiom-co-contra}.
\end{proof}

 Similarly one can use Theorem~\ref{roos-axiom-co-contra}
and~\cite[Theorem~4.8.1]{Pcosh} in order to deduce the Roos axiom
for the category of quasi-coherent sheaves $X\Qcoh$ on any
quasi-compact semi-separated scheme~$X$ (but in this case,
Corollary~\ref{roos-axiom-very-flat-generator-semi-separated}
provides a more direct argument).

 The assertion of Corollary~\ref{roos-axiom-co-contra-finite-krull-dim}
is an existence theorem, and its proof is a nonconstructive existence
proof.
 To extract a specific bound on the number~$n$ appearing in
Corollary~\ref{roos-axiom-co-contra-finite-krull-dim}, one has to
look into the details of the construction of the co-contra
correspondence functors $\Phi$ and $\Psi$
in~\cite[Theorem~6.8.1]{Pcosh}.
 The related argument is presented in~\cite[Theorem~6.10.1]{Pcosh};
it shows that one can have $n=2D+1$, where $D$ is the Krull dimension
of~$X$.

\begin{quest}
 Does the Roos axiom hold for the categories $X\Qcoh$ of
quasi-coherent sheaves on non-semi-separated Noetherian schemes
of infinite Krull dimension?
 All our proofs (in
Theorems~\ref{roos-axiom-cech-coresolution-semi-separated}
and~\ref{roos-axiom-cech-coresolution-finite-krull-dim},
as well as in
Corollaries~\ref{roos-axiom-very-flat-generator-semi-separated}
and~\ref{roos-axiom-co-contra-finite-krull-dim}, and also
in~\cite[Theorem~6.10.1]{Pcosh}) assume either finite Krull
dimension, or semi-separatedness.
 Still, we are not aware of any counterexamples.
\end{quest}

\end{document}